\documentclass[11pt,a4paper]{amsart}

\usepackage[pdftex, 
	a4paper,
	unicode,
	linktocpage,
	urlbordercolor={1 0.5 0.125},
	filebordercolor={1 0.5 0.125},
	linkbordercolor={0.25 1 0.25},
	citebordercolor={0.25 1 0.25},
	pdfusetitle,
	pdftitle={Finitely generated infinite simple groups
	of infinite commutator width and
	vanishing stable commutator length},
	pdfauthor={Alexey Muranov},
	pdfproducer={pdfLaTeX},
	pdflang=EN, 
	pagebackref=false,
	pdfstartview=FitH,
	pdfborderstyle={/S/U/W 1}, 
	breaklinks]{hyperref}

\usepackage[T1]{fontenc}
\usepackage{lmodern}

\usepackage[utf8]{inputenc}

\usepackage{graphicx}
\usepackage{amssymb}

\newtheorem*{theoremA}{Main Theorem}
\newtheorem{theorem}{Theorem}[section]
\newtheorem{proposition}[theorem]{Proposition}
\newtheorem{lemma}[theorem]{Lemma}
\newtheorem{corollary.lemma}[theorem]{Corollary}
\newtheorem{corollary.theorem}[theorem]{Corollary}

\theoremstyle{definition}
\newtheorem*{definition}{Definition}

\theoremstyle{remark}
\newtheorem{remark}[theorem]{Remark}

\newcommand{\ie}{i\textup.e\textup.{}}

\newcommand{\cl}{\operatorname{cl}}
\newcommand{\st}[1]{\overline{#1}}
\newcommand{\stcl}{\st\cl}
\newcommand{\cw}{\operatorname{cw}}
\newcommand{\sql}{\operatorname{sql}}
\newcommand{\sqw}{\operatorname{sqw}}

\providecommand{\abs}[1]{\lvert#1\rvert}
\providecommand{\norm}[1]{\lVert#1\rVert}

\newcommand{\GP}[2]{\langle\,#1\,\Vert\,#2\,\rangle}
\newcommand{\tGP}[2]{$\GP{#1}{#2}$}

\newcommand{\cntr}{\partial}

\begin{document}




\title[Simple groups of infinite square width]
	{Finitely generated infinite simple groups\\
	of infinite square width and vanishing\\
	stable commutator length}

\author{Alexey Muranov}

\date{\today}

\address{Institut de Math\'ematiques de Toulouse\\
	Universit\'e de Toulouse\\
	118 route de Narbonne\\
	F--31062 Toulouse Cedex 9\\
	France}

\email{\href{mailto:muranov@math.univ-toulouse.fr}
	{\nolinkurl{muranov@math.univ-toulouse.fr}}}






\begin{abstract}
It is shown that there exist finitely generated infinite
simple groups of infinite commutator width and infinite square width
on which there exists no stably unbounded conjugation-invariant norm,
and in particular stable commutator length vanishes.
Moreover, a recursive presentation of such a group
with decidable word and conjugacy problems
is constructed.
\end{abstract}


\maketitle
\tableofcontents




\section{Introduction}
\label{section:introduction}

This article is an extension of \cite{Muranov:2007:fgisgicw},
where finitely generated infinite simple
groups of infinite \emph{commutator width\/} were constructed using
a small-cancellation approach.
It is explained here how to modify that construction in order to
make the \emph{square width\/} of the constructed group $G$
also infinite and to satisfy
the additional condition that there exist no
\emph{stably unbounded conjugation-invariant norm\/} on $G$.
This additional condition implies in particular
vanishing of \emph{stable commutator length},
which in turn is equivalent to injectivity of
the natural homomorphism $H_b^2(G,\mathbb R)\to H^2(G,\mathbb R)$
from the second bounded cohomology to
the second ordinary cohomology
(see
Proposition 3.4 in \cite{Bavard:1991:lsc-fr} or
Theorem 3.18 in \cite{Grigorchuk:1995:srbc}).
Also an effort is made to explain the technics of
\cite{Muranov:2007:fgisgicw} more informally.

\begin{definition}
Let $G$ be a group.
The \emph{commutator\/} of elements $x$ and $y$ of $G$ is
$[x,y]=xyx^{-1}y^{-1}$.
The \emph{commutator length\/} of an element $g$ of the derived
subgroup $[G,G]$, denoted $\cl_G(g)$,
is the minimal $n$ such that there exist
$x_1,\dotsc,x_n,y_1,\dotsc,y_n\in G$ such that
$g=[x_1,y_1]\dotsm[x_n,y_n]$.
(Naturally, $\cl(1)=0$.)
The \emph{stable commutator length\/} of $g\in[G,G]$
shall be denoted by $\stcl_G(g)$ and is defined by
$$
\stcl_G(g)=\lim_{n\to\infty}\frac{\cl_G(g^n)}{n}
=\inf_{n\in\mathbb N}\frac{\cl_G(g^n)}{n}.
$$
The \emph{commutator width\/} of $G$ shall be denoted $\cw(G)$
and is defined by
$$
\cw(G)=\sup_{[G,G]}\cl_G.
$$
\end{definition}

Stable commutator length is related to the space of
\emph{homogeneous quasi-morphisms\/} on $G$, and hence
to the kernel of the natural homomorphism
$H_b^2(G,\mathbb R)\to H^2(G,\mathbb R)$,
which is isomorphic to the quotient
of the space of all homogeneous quasi-morphisms $G\to\mathbb R$
by the subspace of all homomorphisms $G\to\mathbb R$,
see Proposition 3.3.1(1) in \cite{Bavard:1991:lsc-fr} or
Theorem 3.5 in \cite{Grigorchuk:1995:srbc}.

\begin{definition}
Let $G$ be a group.
A function $\phi\colon G\to\mathbb R$
is called a \emph{quasi-morphism\/} if
the function $(x,y)\mapsto\phi(xy)-\bigl(\phi(x)+\phi(y)\bigr)$
is bounded on $G\times G$.
A quasi-morphism is \emph{homogeneous\/} if
its restriction to every cyclic subgroup is a homomorphism
to $(\mathbb R,+)$.
\end{definition}

Every quasi-morphism $G\to\mathbb R$
is bounded on the set of all commutators of $G$.
It is also relatively easy to see that if
some homogeneous quasi-morphisms
is non-zero on $g\in[G,G]$, then $\stcl_G(g)>0$,
and in particular $\cw G=\infty$.
Christophe Bavard \cite[Proposition 3.4]{Bavard:1991:lsc-fr}
proved that in fact
stable commutator length vanishes on the whole of the derived subgroup
if and only if so do all homogeneous quasi-morphisms.
(Moreover, Bavard provided a formula for commutator length
in terms of homogeneous quasi-morphisms.)
Furthermore, observe
(or see Proposition 3.3.1(2) in \cite{Bavard:1991:lsc-fr})
that a homogeneous quasi-morphism vanishes on
the derived subgroup only if it is a homomorphism;
therefore, the natural homomorphism
$H_b^2(G,\mathbb R)\to H^2(G,\mathbb R)$
is injective if and only if stable commutator length vanishes
on~$[G,G]$.

A comprehensive introduction to the theory of stable commutator length
may be found in~\cite{Calegari:2009:scl}.

Until 1991 it was apparently not known that there exist
simple groups of commutator width greater than~$1$.
For finite simple groups, it was shown in 2008
by Martin W. Liebeck, Eamonn A.\ O’Brien, Aner Shalev,
and Pham Huu Tiep \cite{LiebeckOST:2010:Oc} that every element
of every non-abelian finite simple group is a commutator,
and thus the long-standing conjecture of
Oystein Ore \cite{Ore:1951:src} was proved.%
\footnote{A stronger conjecture of
John Thompson that every non-abelian finite simple group $G$
has a conjugacy class $C$ such that $G=CC$ still remains open.
Previously, important results towards resolution of the two conjectures
were obtained by Erich Ellers and Nikolai Gordeev in
\cite{EllersGordeev:1998:cTO} and by
A.\ Shalev in~\cite{Shalev:2009:wmccncWtt}.}
Jean Barge and Étienne Ghys
\cite[Theorem 4.3]{BargeGhys:1992:cEM-fr}
showed that there are simple groups of symplectic diffeomorphisms of
$\mathbb R^{2n}$
(kernels of \emph{Calabi homomorphisms\/})
which possess nontrivial homogeneous quasi-morphisms, and thus
their commutator width is infinite.
Theorem 1.1 in \cite{GambaudoGhys:2004:cds} provides other
similar examples of simple groups of infinite commutator width.
Existence of \emph{finitely generated\/} simple groups of commutator width
greater than $1$ was proved in \cite{Muranov:2007:fgisgicw}.
Pierre-Emmanuel Caprace and Koji Fujiwara
\cite{CapraceFujiwara:2010:r1ibqmKMg}
recently proved that
there are \emph{finitely presented\/}
simple groups for which the space of homogeneous quasi-morphisms
is infinite-dimensional,
and in particular whose commutator width is infinite.
Those groups are the quotients of certain non-affine Kac–Moody lattices
by the center;
they were defined by Jaques Tits \cite{Tits:1987:upKMgf}
and their simplicity was proved by P.-E. Caprace
and Bertrand Rémy \cite{CapraceRemy:2006:sagKM-fr}.

Commutator length in a group $G$
is an example of a conjugation-invariant norm
on the derived subgroup $[G,G]$,
as defined by
Dmitri Burago, Sergei Ivanov, and Leonid Polterovich
in~\cite{BuragoIvaPol:2008:cinggo}.

\begin{definition}
A \emph{conjugation-invariant norm\/} on a group $G$ is a function
$\nu\colon G\to[0,\infty)$ which satisfies the following
five axioms:
\begin{enumerate}
\item
\label{item:definition.conjugation-invariant_norm.1}
	$\nu(g)=\nu(g^{-1})$ for all $g\in G$,
\item
\label{item:definition.conjugation-invariant_norm.2}
	$\nu(gh)\le\nu(g)+\nu(h)$ for all $g,h\in G$,
\item
\label{item:definition.conjugation-invariant_norm.3}
	$\nu(g)=\nu(hgh^{-1})$ for all $g,h\in G$,
\item
\label{item:definition.conjugation-invariant_norm.4}
	$\nu(1)=0$,
\item
\label{item:definition.conjugation-invariant_norm.5}
	$\nu(g)>0$ for all $g\in G\setminus\{1\}$.
\end{enumerate}
For brevity, conjugation-invariant norms shall be sometimes called
simply \emph{norms}.
\end{definition}

\begin{definition}
If $\nu$ is a norm on $G$, then
its \emph{stabilization}, which shall be denoted $\st\nu$, is defined by
$$
\st\nu(g)=\lim_{n\to\infty}\frac{\nu(g^n)}{n}
=\inf_{n\in\mathbb N}\frac{\nu(g^n)}{n},\quad g\in G.
$$
A norm $\nu$ is \emph{stably unbounded\/}
if $\st\nu(g)>0$ for some~$g\in G$.
\end{definition}

Note that in general the stabilization of a norm is not a norm, as it
has no reason to satisfy the axioms
(\ref{item:definition.conjugation-invariant_norm.2})
and (\ref{item:definition.conjugation-invariant_norm.5})
of the definition.

The following question was raised by Burago, Ivanov, and Polterovich
\cite{BuragoIvaPol:2008:cinggo}:
\begin{quote}
	Does there exist a group that does not admit a stably unbounded norm
	and yet admits a norm unbounded on some cyclic subgroup?
\end{quote}
The main theorem of this article answers this question positively.

Similarly to commutator length, one can define
\emph{square length}.

\begin{definition}
Let $G$ be a group.
The \emph{square length\/} of an element $g$ of the
subgroup $G^2=\langle\,x^2\,\mid\,x\in G\,\rangle$, denoted $\sql_G(g)$,
is the minimal $n$ such that there exist
$x_1,\dotsc,x_n\in G$ such that
$g=x_1^2\dotsm x_n^2$.
The \emph{square width\/} of $G$ is
$$
\sqw(G)=\sup_{G^2}\sql_G.
$$
\end{definition}

Observe that finite commutator width implies finite square width,
since every commutator is the product of $3$ squares (and since abelian
groups have square width at most $1$).
Moreover,
$\sql_G(x)\le2\cl_G(x)+1$ for every $x\in[G,G]$,
and this estimate cannot be improved
unless $\cl_G(x)=0$ ($x=1$),
see \cite[Section~2.5]{Culler:1981:ussefg}.

\begin{theoremA}
\label{theorem:A.(1)}
There exists a torsion-free simple group\/ $G$ generated by\/ $2$ elements\/
$a$ and\/ $b$ such that\/\textup:
\begin{enumerate}
\item
	$a^2$ and\/ $b^2$ freely generate a free subgroup\/ $H$ such that
$$
	\lim_{n\to\infty}\cl_G(h^n)=\infty
	\quad\text{for every\/ $h\in H\setminus\{1\}$}
$$
	\textup(in particular\textup, $\cw(G)=\infty$\textup)\textup,
\item
	$\sql_G$ is unbounded on\/ $H=\langle a^2,b^2\rangle$\/
	\textup(in particular\textup, $\sqw(G)=\infty$\textup)\textup,
\item
	$G$ does not admit any stably unbounded
	conjugation-invariant norm\/
	\textup(in particular\textup, $\stcl_G=0$\textup)\textup,
\item
	$G$ is the direct limit of a sequence of hyperbolic groups with respect
	to a family of surjective homomorphisms\textup,
\item
	the cohomological and geometric dimensions of\/ $G$
	are\/~$2$\textup,
\item
	$G$ has decidable word and conjugacy problems\textup.
\end{enumerate}
\end{theoremA}

This theorem is stronger than Theorem 5 in \cite{Muranov:2007:fgisgicw}.
It provides positive answer to the question of
Burago-Ivanov-Polterovich
and also shows that stable commutator length can vanish
on a simple group of infinite commutator width
(and even of infinite square width).



\section{The presentation}
\label{section:presentation_construction}

The construction presented here is similar to the one
in \cite[Section~2]{Muranov:2007:fgisgicw}.

In what follows, if $w$ denotes a group word,
then the group element represented by $w$
shall be denoted
by $[w]$, or $[w]_G$ if $G$ is the group in question.

Let $\{a,b\}$ be a $2$-letter alphabet
(one could start here with any finite alphabet containing
at least $2$ letters).
The set of all group words over $\{a,b\}$ shall be denoted by
$\{a^{\pm1},b^{\pm1}\}^*$ or $(\{a,b\}^{\pm1})^*$.
Let $F$ be the free group formally freely generated by $a$ and $b$,
that is the group presented by \tGP{a,b}{\varnothing}.
To prove the main theorem, a group $G$ with desired properties shall be
constructed as a quotient of $F$ by recursively constructing
its presentation \tGP{a,b}{\mathcal R}.

First of all,
fix an arbitrary nontrivial reduced group word $v$ over $\{a,b\}$
(or any group word which is not freely trivial),
and let $C_1, C_2, \dotsc$
be a list of conditions to be satisfied by $G$
composed as follows:
\begin{enumerate}
\item
	for every $w\in\{a^{\pm1},b^{\pm1}\}^*$ and $x\in\{a,b\}$,
	the list contains the condition
	\begin{quote}
		``if $[\underline{w}]\ne1$,
		then $[\underline{x}]$ is the product of conjugates
		of $[\underline{w}]$,''
	\end{quote}
	this condition shall be called a \emph{condition of the first kind},
\item
	for every $w\in\{a^{\pm1},b^{\pm1}\}^*$ and every $n\in\mathbb N$,
	the list contains the condition
	\begin{quote}
		``there exist $p,q\in\mathbb N$ such that
		$p\ge\underline{n}q$ and
		$[\underline{w}]^{p}$ is the product of $q$ conjugates
		of $[\underline{v}]$,''
	\end{quote}
	this condition shall be called a \emph{condition of the second kind},
\item
	the list contains no other conditions.
\end{enumerate}
Observe that all conditions in this list are preserved
under passing to quotients.

Let $z_1=aa=a^2$ and $z_2=bb=b^2$.
The elements $[z_1]_G,[z_2]_G$ are intended to freely generate
a free subgroup $H$ of $G$ such as in the statement of the main theorem
(there are many other possible choices for $z_1,z_2$).
The following properties of $z_1,z_2$
(some necessary, other just convenient)
shall be noted:
\begin{enumerate}
\item
	$[z_1],[z_2]$ freely generate a free subgroup of~$F$,
\item
	$z_1$ and $z_2$ are positive in the sense that they do not
	contain $a^{-1}$ or $b^{-1}$, and hence any concatenation
	of copies of $z_1,z_2$ is cyclically reduced,
\item
	every element of the free subgroup $\langle[z_1],[z_2]\rangle$
	of $F$ can be written as a reduced concatenation of copies of
	$z_1^{\pm1},z_2^{\pm1}$, and is conjugate to an element
	that can be written as a cyclically reduced
	concatenation of copies of
	$z_1^{\pm1},z_2^{\pm1}$,
\item
\label{item:z_1_z_2.properties.4}
	for every $\varepsilon>0$ and every $L>0$, there exists
	a non-periodic
	(and hence not representing a proper power in $F$)
	concatenation $w$ of copies of
	$z_1,z_2$ of length at least $L$ such that
	if $sq_1$ and $sq_2$ are two distinct cyclic shifts of $w$,
	then $\abs{s}\le\varepsilon\abs{w}$.
\end{enumerate}
To verify the last property, one can consider, for example,
$$
w=\prod_{i=0}^{n}a^{2i}b^{2n-2i}
$$
with $n$ sufficiently large.

The desired indexed family of defining relators
$\mathcal R=\{r_n\}_{n\in I}$
shall be constructed in two steps.
First, a family $\mathcal Q=\{r_n\}_{n\in\mathbb N}$
shall be constructed so that for all $n$, the condition $C_n$
be a consequence of the relation $\ulcorner r_n=1\urcorner$, and hence
the group presented by \tGP{a,b}{\mathcal Q} shall satisfy all of
the conditions $C_1, C_2, \dotsc$;
however, this group will be the trivial group.
Second, certain redundant relators $r_n$ corresponding to
conditions $C_n$ of the first kind
shall be removed from $\mathcal Q$ to obtain
$\mathcal R=\{r_n\}_{n\in I}$, $I\subset\mathbb N$.

The construction of $\mathcal Q$
can easily be carried out effectively, so that $\mathcal Q$ be recursive.
In the process of establishing decidability of the word and conjugacy
problems in $G=\GP{a,b}{\mathcal R}$,
it shall be shown that $\mathcal R$ can be made recursive as well
(all depends on a good choice of~$\mathcal Q$).

The defining relators
shall be chosen to satisfy certain small-cancel\-lation
conditions to ensure that the obtained group $G$ be nontrivial,
of infinite commutator and square widths,
with commutator length unbounded on
nontrivial cyclic subgroups of $H=\langle[z_1]_G,[z_2]_G\rangle$, etc.
Additionally, these small-cancel\-lation
conditions will cause all finite subpresentations of \tGP{a,b}{\mathcal R}
to define hyperbolic groups.

The idea of the proof that commutator and square lengths are
unbounded is to find for every $N<0$ some group word
$w$ such that
the Euler characteristic of every
van Kampen diagram over \tGP{a,b}{\mathcal R}
with the boundary label $w$ be necessarily less than $N$.
The defining relations shall be chosen so that
``most'' sufficiently long reduced concatenation of copies of
$z_1^{\pm1}$ and $z_2^{\pm1}$ suit for the role of $w$ here.

Here follows the formal construction of $\mathcal Q$ and~$\mathcal R$.

Let $A$ be the set of all $n$ such that the condition $C_n$ is of
the first kind.
For every $n\in A$,
let $w_n\in\{a^{\pm1},b^{\pm1}\}^*$ and $x_n\in\{a,b\}$ be such that
$C_n$ states:
\begin{quote}
	``if $[\underline{w_n}]\ne1$, then $[\underline{x_n}]$ is the product
	of conjugates of $[\underline{w_n}]$,''
\end{quote}
the relator $r_n$ then shall be chosen in the form
$$
r_n=u_{n,1}w_nu_{n,1}^{-1}\dots u_{n,k_n}w_nu_{n,k_n}^{-1}x_n^{-1}
=w_n^{u_{n,1}}\dots w_n^{u_{n,k_n}}x_n^{-1},
$$
where $u_{n,1},\dotsc,u_{n,k_n}\in\{a^{\pm1},b^{\pm1}\}^*$ and $k_n\ge3$;
such $r_n$ shall be called a \emph{relator of the first kind}.
For every $n\in\mathbb N\setminus A$,
let $w_n\in\{a^{\pm1},b^{\pm1}\}^*$ and $m_n\in\mathbb N$ be such that
$C_n$ states:
\begin{quote}
	``there exist $p,q\in\mathbb N$ such that
	$p\ge\underline{m_n}q$ and
	$[\underline{w_n}]^{p}$ is the product of $q$ conjugates
	of $[\underline{v}]$,''
\end{quote}
the relator $r_n$ then shall be chosen in the form
$$
r_n=u_{n,1}vu_{n,1}^{-1}\dots u_{n,k_n}vu_{n,k_n}^{-1}(w_n^{m_nk_n})^{-1}
=v^{u_{n,1}}\dots v^{u_{n,k_n}}w_n^{-m_nk_n},
$$
where $u_{n,1},\dotsc,u_{n,k_n}\in\{a^{\pm1},b^{\pm1}\}^*$ and $k_n\ge3$;
such $r_n$ shall be called a \emph{relator of the second kind}.
Observe that the group presented by \tGP{a,b}{\mathcal Q} will be trivial
because some of the defining relations of the first kind
will make both generators $[a]$ and $[b]$ equal to products
of conjugates of~$1$.

The family $\mathcal R=\{r_n\}_{n\in I}$ shall be obtained
from $\mathcal Q=\{r_n\}_{n\in\mathbb N}$
by discarding certain redundant relators of the first kind.
Namely, let the set of indices $I\subset\mathbb N$ be defined
inductively as follows:
\begin{quote}
	for every $n\in\mathbb N$, $n\notin I$ if and only if $n\in A$
	(\ie\ $r_n$ is of the first kind)
	and $[w_n]=1$ in the group presented by \tGP{a,b}{r_i,\ i\in I, i<n}.
\end{quote}
Observe that $I$ is infinite
(because there are infinitely many relators of the second kind
in~$\mathcal Q$).

To finalize the construction of $\mathcal Q$ and $\mathcal R$,
it is left
to specify how to choose the sequence $\{k_n\}_{n\in\mathbb N}$
and the indexed family $\{u_{n,i}\}_{n\in\mathbb N;i=1,\dotsc,k_n}$.

Roughly speaking, the main requirements shall be that
the integers $k_n$ tend to infinity,
that the words $u_{n,i}$ be reduced, ``very long,'' and have
``very short'' common subwords, and
also it is important that for every $n$,
the words $u_{n,1},\dotsc,u_{n,k_n}$
be ``more or less of the same length''
(for further convenience they may be chosen of the same length:
$\abs{u_{n,1}}=\dotsb=\abs{u_{n,k_n}}$).

The sequence $\{k_n\}_{n\in\mathbb N}$
and the family $\{u_{n,i}\}_{n\in\mathbb N;i=1,\dotsc,k_n}$ shall be
chosen simultaneously with four other sequences:
a sequence of integers $\{\chi_n\}_{n\in\mathbb N}$
tending to $-\infty$
(to be used as upper bounds on Euler characteristics)
and three sequences of ``small'' positive reals
$\{\lambda_n\}_{n\in\mathbb N}$,
$\{\mu_n\}_{n\in\mathbb N}$, and
$\{\nu_n\}_{n\in\mathbb N}$.
Also it will be convenient to use two auxiliary sequences
$\{\kappa_n\}_{n\in\mathbb N}$ and $\{\gamma_n\}_{n\in\mathbb N}$
defined by
$$
\kappa_n=2k_n\quad\text{and}\quad
\gamma_n=\lambda_n+(3+2\kappa_n)\mu_n+2\nu_n
$$
for all $n\in\mathbb N$.
The following 12 conditions shall be satisfied:
\begin{enumerate}
\renewcommand{\labelenumi}{(C\theenumi)}
\item
\label{item:presentation_construction.main_properties.1}
	for every $n\in\mathbb N$, $k_n\ge3$,
\item
\label{item:presentation_construction.main_properties.2}
	for every $n\in\mathbb N$ and $i=1,\dotsc,k_n$,
	$u_{n,i}$ is reduced,
\item
\label{item:presentation_construction.main_properties.3}
	for every $n\in\mathbb N$,
	$2(\abs{u_{n,1}}+\dotsb+\abs{u_{n,k_n}})\ge(1-\lambda_n)\abs{r_n}$,
\item
\label{item:presentation_construction.main_properties.4}
	for every $n\in\mathbb N$ and $i=1,\dotsc,k_n$,
	$\abs{u_{n,i}}\le\nu_n\abs{r_n}$,
\item
\label{item:presentation_construction.main_properties.5}
	for every $n_1,n_2\in\mathbb N$,
	$i_1=1,\dotsc,k_{n_1}$ and $i_2=1,\dotsc,k_{n_2}$,
	if $u_{n_1,i_1}^{\sigma_1}=p_1sq_1$ and
	$u_{n_2,i_2}^{\sigma_2}=p_2sq_2$ with
	$\sigma_1,\sigma_2\in\{\pm1\}$,
	then either
	$$
	(n_1,i_1,\sigma_1,p_1,q_1)=(n_2,i_2,\sigma_2,p_2,q_2),
	$$
	or
	$$
	\mu_{n_1}\abs{r_{n_1}}\ge\abs{s}\le\mu_{n_2}\abs{r_{n_2}},
	$$
\item
\label{item:presentation_construction.main_properties.6}
	for every $n\in\mathbb N$ and $i=1,\dotsc,k_n$,
	if $s$ is a common subword of $u_{n,i}^{\pm1}$
	and of a concatenation of several copies of
	$z_1^{\pm1},z_2^{\pm1}$, then
	$$
	\abs{s}\le\mu_n\abs{r_n},
	$$
\item
\label{item:presentation_construction.main_properties.7}
	$$
	\lim_{n\to\infty}\chi_n=-\infty,
	$$
\item
\label{item:presentation_construction.main_properties.8}
	for every $n\in\mathbb N$,
	\begin{equation*}
	\label{display:main_inequality}
	\gamma_n<\frac{1}{2}
	\quad\text{and}\quad
	(3-3\chi_n)\mu_n+(1-\chi_n)\nu_n<\frac{1}{2}-\gamma_n,
	\end{equation*}
\item
\label{item:presentation_construction.main_properties.9}
	for every $n\in\mathbb N$ and every $i<n$ such that $i\in A$
	(\ie\ $r_i$ is of the first kind),
	$$
	\abs{w_i}<(1-2\gamma_n)\abs{r_n},
	$$
\item
\label{item:presentation_construction.main_properties.10}
	for every $n\in\mathbb N$,
	$$
	\abs{v}<(1-2\gamma_n)\abs{r_n},
	$$
\item
\label{item:presentation_construction.main_properties.11}
	the normal closures of the elements $[r_1],[r_2],\dotsc$
	in the free group $F$ (presented by \tGP{a,b}{\varnothing})
	are pairwise distinct,
\item
\label{item:presentation_construction.main_properties.12}
	the elements $[r_1],[r_2],\dotsc$ are not proper powers in~$F$.
\end{enumerate}
In fact,
conditions (C\ref{item:presentation_construction.main_properties.11})
and (C\ref{item:presentation_construction.main_properties.12})
can be deduced from the previous ones,
but it is easier to deduce them from a suitable choice of
$\{k_n\}_{n\in\mathbb N}$,
$\{u_{n,i}\}_{n\in\mathbb N;i=1,\dotsc,k_n}$, etc.

All of these conditions may be fulfilled as follows.
First, choose an arbitrary sequence
of integers $\{k_n\}_{n\in\mathbb N}$ such that
$k_n\ge3$ for all $n$ and
$$
\lim_{n\to\infty}k_n=\infty.
$$
Then, for every $n$, set
$$
\nu_n=\frac{1}{\kappa_n}=\frac{1}{2k_n},\quad\chi_n=4-k_n.
$$
Observe that
$$
\lim_{n\to\infty}\chi_n=-\infty,
$$
and that for every $n$,
$$
2\nu_n<\frac{1}{2}\quad\text{and}\quad(3-\chi_n)\nu_n<\frac{1}{2}.
$$
Now, for every $n$,
choose $\lambda_n$ and $\mu_n$ sufficiently small so that
$$
\gamma_n<\frac{1}{2}
\quad\text{and}\quad
\gamma_n+(3-3\chi_n)\mu_n+(1-\chi_n)\nu_n<\frac{1}{2}.
$$
Finally, given that $z_1=a^2$ and $z_2=b^2$,
the family $\{u_{n,i}\}_{n\in\mathbb N; i=1,\dots,k_n}$
may be chosen in the following form:
$$
u_{n,i}=\prod_{j=2M_n(i-1)+1}^{2M_ni}a^{j}b^{2M_nk_n+1-j},
\label{display:u_and_z.sgicw.def}
$$
where $\{M_n\}_{n\in\mathbb N}$ is a suitable
sequence of ``large'' positive integers ``rap\-id\-ly'' growing
to infinity
(compare with formula (4) in \cite{Muranov:2007:fgisgicw}).
Observe that:
\begin{enumerate}
\item
	$\abs{u_{n,1}}=\dotsb=\abs{u_{n,k_n}}=2M_n(2M_nk_n+1)$ for every $n$,
\item
	if $u_{n_1,i_1}=p_1sq_1$,
	$u_{n_2,i_2}=p_2sq_2$, and
	$(n_1,i_1,p_1,q_1)\ne(n_2,i_2,p_2,q_2)$,
	then $s$ is a subword of
	$$
	a^{j}\,b^{2M_{n_1}k_{n_1}+1-j}\,a^{j+1}
	\quad\text{or}\quad
	b^{2M_{n_1}k_{n_1}+1-j}\,a^{j+1}\,b^{2M_{n_1}k_{n_1}-j}
	$$
	for some $j$
	(because $2M_{n_1}k_{n_1}+1\ne 2M_{n_2}k_{n_2}+1$ unless $n_1=n_2$),
	and hence $\abs{s}\le 4M_{n_1}k_{n_1}+1$,
\item
	if $s$ is a common subword of $u_{n,i}$
	and of a concatenation of several copies of
	$a^2$ and $b^2$,
	then $s$ is a subword of
	$$
	a^{j}\,b^{2M_nk_n+1-j}\,a^{j+1}\,b^{2M_nk_n-j}
	\quad\text{or}\quad
	b^{2M_nk_n+1-j}\,a^{j+1}\,b^{2M_nk_n-j}\,a^{j+2}
	$$
	for some $j$
	(because $2M_nk_n+1$ is odd),
	and hence $\abs{s}\le 4M_nk_n+4$.
\end{enumerate}

Thus the families $\mathcal Q$, $\mathcal R$,
$\{w_n\}_{n\in A}$, $\{x_n\}_{n\in A}$,
$\{w_n\}_{n\in\mathbb N\setminus A}$, $\{m_n\}_{n\in\mathbb N\setminus A}$,
$\{k_n\}_{n\in\mathbb N}$, $\{u_{n,i}\}_{n\in\mathbb N,i=1,\dotsc,k_n}$,
$\{\chi_n\}_{n\in\mathbb N}$, $\{\kappa_n\}_{n\in\mathbb N}$,
$\{\lambda_n\}_{n\in\mathbb N}$, $\{\mu_n\}_{n\in\mathbb N}$,
$\{\nu_n\}_{n\in\mathbb N}$, $\{\gamma_n\}_{n\in\mathbb N}$
and the set $I$
have been constructed.
(Recall that their construction depended on the choice of the words
$v$, $z_1$, $z_2$, and on the list of conditions $C_1,C_2,\dotsc$.)
Let, from now on, $G$ denote the group presented by \tGP{a,b}{\mathcal R}.

\begin{proposition}
\label{proposition:simple.norms_stably_bounded.(2.1)}
If\/ $G$ is not trivial\textup, then it is simple\textup.
There is no stably unbounded conjugation-invariant
norm on\/~$G$\textup.
\end{proposition}
\begin{proof}
Because of the imposed relations of the first kind,
the normal closure of every nontrivial element of $G$
contains both generators $[a]$ and $[b]$,
and hence is the whole of~$G$.

Consider an arbitrary conjugation-invariant norm $\theta$ on $G$.
It follows from the imposed relations
of the second kind that for every $g\in G$
and every $\varepsilon>0$, there exists
$p\in\mathbb N$ such that $\theta(g^p)\le\varepsilon p\theta([v])$.
Hence for every $g\in G$,
$$
\inf_{p\in\mathbb N}\frac{\theta(g^p)}{p}=0.
$$
\end{proof}

In order for $G$ to have decidable word and conjugacy problems,
the presentation \tGP{a,b}{\mathcal R} shall be constructed in a less
random way, so as in particular to be recursive.

\begin{lemma}
\label{lemma:effective_presentation.(2.2)}
The construction described above in this section can be carried out in
such a manner that the following additional properties be satisfied\/\textup:
\begin{enumerate}
\item
\label{item:lemma.effective_presentation.1}
	$\{k_n\}_{n\in\mathbb N}$\textup,
	$\{\chi_n\}_{n\in\mathbb N}$\textup, and
	$\{\lambda_n\}_{n\in\mathbb N}$\textup,
	$\{\mu_n\}_{n\in\mathbb N}$\textup,
	$\{\nu_n\}_{n\in\mathbb N}$
	are recursive sequences of integer and rational
	numbers\textup, respectively\textup,
\item
\label{item:lemma.effective_presentation.2}
	the set $A$ and the families $\mathcal Q$\textup,
	$\{w_n\}_{n\in A}$\textup, $\{x_n\}_{n\in A}$\textup,
	$\{w_n\}_{n\in\mathbb N\setminus A}$\textup,\linebreak[4]
	$\{m_n\}_{n\in\mathbb N\setminus A}$\textup,
	$\{u_{n,i}\}_{n\in\mathbb N,i=1,\dotsc,k_n}$ are recursive\textup,
	and the sequence\linebreak[4]
	$\{(1-2\gamma_n)\abs{r_n}\}_{n\in\mathbb N}$ is bounded
	from below by a recursive sequence of integers
	monotonically tending to\/ $\infty$\textup,
\item
\label{item:lemma.effective_presentation.3}
	if there is an algorithm that decides the word problem for all
	finite subpresentations of\/ \tGP{a,b}{\mathcal R}
	\textup(the input consisting of a finite part of the indexed family
	$\mathcal R$ and
	a pair of words to compare\/\textup)\textup, then
	the set\/ $I$ and the family\/ $\mathcal R$ are recursive\textup.
\end{enumerate}
\end{lemma}
\begin{proof}
To satisfy (\ref{item:lemma.effective_presentation.1}),
it suffices to define $k_n=3+n$ for all $n$, and to choose the four other
sequences accordingly in an effective manner.

To satisfy the rest, the list $C_1, C_2, \dotsc$ should be enumerated so
as to be recursive.
Then the family $\{u_{n,i}\}_{n\in\mathbb N,i=1,\dotsc,k_n}$ can
be constructed effectively and so that
$(1-2\gamma_n)\abs{r_n}\ge n$ for all $n$,
and hence (\ref{item:lemma.effective_presentation.2}) be satisfied.

Suppose there is a way to effectively solve the word problem for
all finite subpresentations of \tGP{a,b}{\mathcal R}.
Then, once (\ref{item:lemma.effective_presentation.2}) is satisfied, and
hence the indexed family $\mathcal Q$ is recursive,
the set $I$ can be constructed inductively
by deciding for each $n\in A$, whether
$[w_n]=1$ in the group presented by
\tGP{a,b}{r_i,\ i\in I, i<n}, and hence concluding whether or not
$n\in I$.
If the family $\mathcal Q$ and the set $I$ are recursive,
then $\mathcal R$ is recursive as well, and
(\ref{item:lemma.effective_presentation.3}) is satisfied.
\end{proof}

\begin{remark}
\label{remark:recursive_presentation__family_vs_set}
Note that in items
(\ref{item:lemma.effective_presentation.2})
and
(\ref{item:lemma.effective_presentation.3})
of Lemma \ref{lemma:effective_presentation.(2.2)},
nothing is stated explicitly about recursiveness of
$\mathcal Q$ and $\mathcal R$ as sets,
\ie\ whether the sets
$\{\,r_n\,\mid\,n\in\mathbb N\,\}$ and
$\{\,r_n\,\mid\,n\in I\,\}$ are recursive.
However, if the sequence $\{\abs{r_n}\}_{n\in\mathbb N}$
is bounded from below by a recursive sequence of integers
monotonically tending to\/ $\infty$, and if the indexed families
$\mathcal Q$ and $\mathcal R$ are recursive
(as functions $\mathbb N\to\{a^{\pm1},b^{\pm1}\}^*$ and
$I\to\{a^{\pm1},b^{\pm1}\}^*$, respectively),
then the corresponding sets
$\{\,r_n\,\mid\,n\in\mathbb N\,\}$ and
$\{\,r_n\,\mid\,n\in I\,\}$ are recursive as well.
\end{remark}



\section{Maps and diagrams}
\label{section:maps-diagrams}

The terminology related to graphs, combinatorial complexes,
maps, and diagrams used in this paper shall be similar to the one
in \cite{Muranov:2007:fgisgicw}, but there shall be a few differences.
Definitions of maps and diagrams in a more topological
(less combinatorial) way may be found in
\cite{LyndonSchupp:2001:cgt,
Olshanskii:1989:gosg-rus,
Olshanskii:1991:gdrg-eng}.

\subsection{Combinatorial complexes}
\label{subsection:maps-diagrams.combinatorial_complexes}

Roughly speaking, a \emph{combinatorial complex\/} shall be viewed as
a combinatorial model of a CW-complex.
As defined in \cite{Muranov:2007:fgisgicw}, combinatorial complexes
are somewhat restricted in their capability of representing CW-complexes,
for example they cannot represent
a CW-complex in which the whole characteristic boundary of some $2$-cell
is attached to a single $0$-cell, but otherwise,
at least in the $2$-dimensional case,
they adequately represent the combinatorics of CW-complexes.
The CW-complex represented by a given combinatorial complex
shall be called its \emph{geometric realization}.
Standard terminology of CW-complexes may also be used for
combinatorial complexes, as the analogy is usually clear.

The structure of each cell in a combinatorial complex comprises its
\emph{characteristic boundary\/} and its
\emph{attaching morphism}.
The characteristic boundary of a $0$-cell is empty,
the characteristic boundary of a $1$-cell is a complex consisting
of $2$ vertices,
the characteristic boundary of a $2$-cell is a
\emph{combinatorial circle}.
The attaching morphism of an $n$-cell $x$ of
a combinatorial complex $\Delta$
is a morphism from the characteristic boundary of $x$ to
the $(n-1)$-skeleton of $\Delta$.
The set of all $n$-cells of a complex $\Delta$
shall be denoted by~$\Delta(n)$.

In $2$-dimensional combinatorial complexes, $0$-cells shall be called
\emph{vertices}, $1$-cells shall be called \emph{edges},
and $2$-cells shall be called \emph{faces}.

A \emph{path\/} in a combinatorial complex is
an alternating sequence of vertices and oriented edges in which every
oriented edges is immediately preceded by its initial vertex and
immediately succeeded by its terminal vertex.
A path of length $0$ (consisting of a single vertex)
shall be called \emph{trivial}.
For simplicity, a nontrivial path may be written
as the sequence of its oriented edges.
A path is \emph{closed\/} if its \emph{initial\/} and \emph{terminal\/}
vertices coincide.
A path is \emph{reduced\/} if it does not contain a subpath
of the form $ee^{-1}$ where $e$ is an oriented edge.
A closed path is \emph{cyclically reduced\/} if
it is reduced and its first
oriented edge is not inverse to its last oriented edge.
A path is \emph{simple\/} if
it is nontrivial, reduced, and
none of its \emph{intermediate\/} vertices appears in it more than once
(a closed path can be simple).

Two combinatorial complexes shall be called
\emph{combinatorially equivalent\/}
if they possess isomorphic \emph{subdivisions}.

A $2$-dimensional combinatorial complex whose geometric realization
is a surface (with or without boundary) shall be called a
\emph{combinatorial surface}.

If $\Delta$ is a combinatorial complex, then its
\emph{Euler characteristic\/} shall be denoted by~$\chi(\Delta)$.

In what follows, the word ``combinatorial''
shall sometimes be omitted for brevity.
For example, a \emph{sphere\/} shall also mean a
\emph{combinatorial sphere},
that is a complex whose geometric realization is a sphere,
or equivalently a complex combinatorially equivalent
to any particular ``sample'' combinatorial sphere.

\subsection{Maps}
\label{subsection:maps-diagrams.maps}

A \emph{nontrivial connected map\/} $\Delta$
consists of the following data:
\begin{enumerate}
\item
	a nontrivial (\ie\ not consisting of a single vertex)
	non-empty finite connected combinatorial complex, which shall also be
	denoted $\Delta$ if no confusion arises;
\item
	a finite collection of combinatorial circles
	together with their morphisms to $\Delta$,
	which shall be called the \emph{characteristic boundary\/}
	of~$\Delta$;
\item
	a choice of simple closed paths in the characteristic
	boundaries of all faces of $\Delta$ and in
	all components of the characteristic boundary
	of $\Delta$, called the \emph{characteristic contours\/}
	of faces and \emph{characteristic contours\/} of $\Delta$,
	respectively,
	\ie\ to every face there should be assigned its characteristic contour,
	and in every component of the characteristic boundary of $\Delta$,
	there should be chosen one characteristic contour of $\Delta$;
	moreover, all characteristic contours of $\Delta$
	shall be enumerated for further convenience.
\end{enumerate}
The image of the characteristic contour of a face in the $1$-skeleton
of $\Delta$ (under the attaching morphism of that face)
shall be called the \emph{contour\/} of that face,
and the image of a characteristic contour of $\Delta$ shall be called
a \emph{contour\/} of $\Delta$ and shall be enumerated with
the same number.
The only requirement to this structure is that
attaching a new face along each of the contours of $\Delta$
must turn $\Delta$ into a closed combinatorial surface.

A \emph{trivial disc map}, or simply a \emph{trivial map},
is defined as
a trivial (\ie\ consisting of a single vertex) combinatorial complex
together with another trivial complex mapped to it in the role of its
\emph{characteristic boundary\/} and together with the corresponding
trivial \emph{characteristic contour\/} and
also trivial \emph{contour}.
(A trivial disc map cannot be turned into a combinatorial sphere
by attaching a face along its contour because a face cannot
be attached to a single vertex.)

In general a \emph{map\/} consists of a finite number of connected
components each of which is either a nontrivial connected map
or a trivial disc map, with the assumption that all characteristic
contours are enumerated by different numbers.

The contour of a face $\Pi$ shall be denoted by $\cntr\Pi$.
The contours of a map $\Delta$ shall be denoted by
$\cntr_1\Delta,\cntr_2\Delta$, etc., but
if $\Delta$ has only one contour, then it may be
denoted~$\cntr{}\Delta$.

Note that it is more usual to fix
contours and characteristic contours only up to cyclic shifts.
The definition given above is chosen because it is
convenient to have
each contour defined as a specific path, rather than a ``cyclic path,''
to have a ``base point'' chosen in the characteristic boundary
of each face and in each component of the characteristic boundary
of the map,
and because in the context of \emph{diagrams\/} (defined below),
it is convenient to be able to label contours with specific group words,
rather than ``cyclic group words.''

It is useful to note that a map can be reconstructed
up to isomorphism and choice of characteristic contours
from its characteristic boundary
and the characteristic boundaries of its faces if
for each pair of oriented edges taken
from the characteristic boundary of the map and
the characteristic boundaries of its faces
is known whether they are mapped to the same oriented edge
in the $1$-skeleton of the map.

A map is \emph{closed\/} if it has no contours,
or, equivalently, if its underlying complex is a closed surface.
If a map $\Delta$ has no trivial disc maps as connected components
(which shall be called its \emph{trivial connected commponents\/}),
then the \emph{closure\/} of $\Delta$
is the closed map $\bar\Delta$ obtained from $\Delta$
by naturally turning each component of
the characteristic boundary of $\Delta$ into a new ``outer'' face,
which may be viewed as attaching new faces along all contours of $\Delta$
and choosing the characteristic contours of the attached faces
so that their contours coincide respectively
with the contours of~$\Delta$.

A map shall be called \emph{orientable\/}
if the closure of each of its nontrivial connected
components is orientable.
To \emph{orient\/} a map shall mean to orient all of its faces,
which means to choose \emph{positive\/} and \emph{negative\/}
directions in their characteristic boundaries,
and to choose \emph{positive\/} and \emph{negative\/}
directions in all components of the characteristic boundary
of the map
so that this choice would induce
an orientation of the closures of all nontrivial connected components.
Nontrivial contours and characteristic contours
shall be called \emph{positive\/} or \emph{negative\/} accordingly,
and positive contours shall be said to \emph{agree\/}
with its orientation.

A \emph{singular combinatorial disc\/} is any
simply connected non-empty subcomplex of a combinatorial sphere.%
\footnote{The word ``singular'' here apparently is not related in meaning
to the same word as used in \cite{CollinsHuebschmann:1982:sdir}
or~\cite{LyndonSchupp:2001:cgt}.}
A \emph{disc map\/} is a map whose underlying complex
is a singular combinatorial disc.
Alternatively, a disc map can be defined as a connected one-contour map
of Euler characteristic~$1$.
The closure of a nontrivial disc map is a spherical map.

A map is \emph{simple\/} if its contours are
pairwise disjoint simple closed paths.

An \emph{elementary spherical map\/} is
a spherical map with exactly $2$ faces
whose $1$-skeleton is a combinatorial circle.
Such maps are the least interesting ones.

The following lemma follows from the classification of
compact (or finite combinatorial) surfaces.

\begin{lemma}
\label{lemma:positive_Euler_characteristic.(3.1)}
The maximal possible Euler characteristic of a closed connected
combinatorial surface is\/ $2$\textup, and among all
such surfaces\textup,
only spheres have Euler characteristic\/ $2$\textup, and
only projective planes have Euler characteristic~$1$.
The maximal possible Euler characteristic of a
proper connected subcomplex of a combinatorial
surface is\/ $1$\textup, and
every such subcomplex is a singular combinatorial disc\textup.
\end{lemma}

\begin{corollary.lemma}
\label{corollary.lemma:maps__positive_Euler_characteristic.(3.2)}
The maximal possible Euler characteristic of a connected map with\/
$n$ contours is\/~$2-n$\textup.
A connected map of Euler characteristic\/ $2$ is spherical\textup.
A closed connected map of Euler characteristic\/ $1$
is projective-planar\textup.
A non-closed connected map of Euler characteristic\/ $1$ is disc\textup.
\end{corollary.lemma}

Every subcomplex $\Gamma$ of a map $\Delta$
can be endowed with a structure of a \emph{submap},
which is unique up to choice and enumeration of
characteristic contours of $\Gamma$
(characteristic contours of faces of $\Gamma$
can be taken from $\Delta$).
Note that the characteristic boundary of $\Gamma$ is uniquely
determined, and hence so is the underlying complex of its closure.

\subsection{Diagrams}
\label{subsection:maps-diagrams.diagrams}

Maps have been introduced to be used as underlying
objects of \emph{diagrams\/} over group presentations.

Recall that a \emph{group presentation\/} \tGP{X}{R}
consists of an alphabet $X$ and an indexed family $R$
of group words over $X$ (words over $X^{\pm1}$)
called \emph{defining relators}.
A distinction shall be made between
\emph{elements\/} of $R$
and \emph{indexed members\/} of $R$:
an indexed member is an element together with its index in $R$,
hence the same relator can be included in $R$ multiple times as
distinct indexed members.
(Using indexed families of realtors instead of sets is convenient in
the case when a priori it is not known if there are any
repetitions in the imposed relations, and also
for constructing certain geometric object from a given
group presentation.)
The group presented by \tGP{X}{R} is formally generated by $X$
subject to the \emph{defining relations\/}
$\ulcorner r=1\urcorner$, $r\in R$.

If \tGP{X}{R} is a group presentation where $R$ is an indexed
family of nontrivial group words, then
the \emph{combinatorial realization\/} $K(X;R)$ of \tGP{X}{R}
shall be defined as a combinatorial complex with $1$ vertex,
$\norm{A}$ edges (loops), $\norm{R}$ faces, and additional structure
consisting of
\begin{enumerate}
\item
	a bijection between the set of oriented edges of $K(X;R)$
	and the set $X^{\pm1}$ such that mutually inverse oriented
	edges be associated with mutually inverse group letters;
	the group letter associated to an oriented edge shall be called
	the \emph{label\/} of that oriented edge;
\item
	a bijection between the set of faces of $K(X;R)$
	and the set of indexed members of $R$;
	the indexed members of $R$ associated to a face shall be called
	the \emph{label\/} of that face;
\item
	a choice of a simple closed path in the characteristic boundary
	of each face of $K(X;R)$,
	which shall be called the \emph{characteristic contour\/}
	of that face,
	and whose image in the $1$-skeleton of $K(X;R)$
	shall be called the \emph{contour\/}
	of that face;
\end{enumerate}
this structure shall satisfy the condition that the contour label
of each face shall coincide with the label of that face viewed as
a group word.
The group presented by \tGP{X}{R} is naturally isomorphic
to the fundamental group of~$K(X;R)$.

A \emph{diagram\/} over a group presentation \tGP{X}{R} is
a map in which
all oriented edges are labeled with elements of $X^{\pm1}$
and all faces are labeled with indexed relators from $R$
so that mutually inverse oriented edges are
labeled with mutually inverse group letters and
if a face $\Pi$ is labeled with
an indexed relator $r$, then the contour label of $\Pi$
is the group word $r$.
Labeling of a diagram $\Delta$ over \tGP{X}{R}
is equivalent to defining
a morphism from its underlying map to $K(X;R)$
that preserves the characteristic contours of faces.

In a diagram or a combinatorial realization of a presentation,
the labels of an oriented edge $e$, a path $p$,
and a face $\Pi$
shall be denoted by $\ell(e)$, $\ell(p)$, and $\ell(\Pi)$, respectively.

\subsubsection{Augmented diagrams}
\label{subsubsection:maps-diagrams.diagrams.augmented_diagrams}

An \emph{augmented diagram\/} over \tGP{X}{R}, also
\linebreak[4]
called a \emph{$0$-refined diagram\/}
in \cite{Olshanskii:1991:gdrg-eng},
is, similarly to a regular diagram, a kind of a labeled map,
but in which edges and faces are divided
into \emph{regular\/} and \emph{auxiliary}.
Regular edges shall be called \emph{$1$-edges}, and regular faces
shall be called \emph{$2$-faces}.
Auxiliary edges shall be called \emph{$0$-edges}, and
auxiliary faces shall be divided into
\emph{$0$-faces\/} and \emph{$1$-faces}.
Regular oriented edges shall be labeled with elements
of $X^{\pm1}$, and regular faces shall be labeled with
indexed members of $R$ following the same rules as for
regular diagrams.
All oriented $0$-edges shall be labeled with the symbol~$1$.
A $0$-face can be incident only to $0$-edges, and
hence the contour label of every $0$-face has to be of the form $1^k$,
$k\in\mathbb N$.
The contour label of every $1$-face
has to be a cyclic shift of $x1^kx^{-1}1^l$ for some $x\in X$,
$k,l\in\mathbb N\cup\{0\}$.
Thus the contour label of every face of an augmented diagram
over \tGP{X}{R}
is either an element of $R$ or freely trivial.

Regular diagrams shall be viewed as augmented diagrams without
auxiliary edges or faces, but a \emph{diagram\/} by default
shall mean a regular diagram.

\subsubsection{Regularization of augmented diagrams}
\label{subsubsection:maps-diagrams.diagrams.regularization}

There is a procedure to canonically transform any augmented diagram
into a regular one which shall be called its
\emph{regularization}.
Essentially, this procedure consists in collapsing all
$0$-edges and $0$-faces into
vertices, and all $1$-faces onto incident $1$-edges;
there are however a few special cases.
Here follows a more formal definition of this transformation.

Let $\Delta$ be an augmented diagram.
To \emph{regularize\/} $\Delta$, the following steps shall be done
(in the given order):
\begin{enumerate}
\item
\label{item:regularization.1}
	for each contour of $\Delta$ that lies entirely on $0$-edges,
	this contour shall be filled in with a $0$-face, and a trivial
	connected component shall be added, whose contour shall replace
	the filled-in contour;
\item
\label{item:regularization.2}
	for each contour that does not lie entirely on $0$-edges
	but not entirely on $1$-edges either,
	a ``collar'' of $0$- and $1$-faces shall be ``glued'' to this
	contour so as to ``shift'' it to a new contour that lies
	entirely on $1$-edges and whose label is obtained from
	the label of the original contour by canceling all $1$ symbols;
	after this operation,
	all nontrivial contours lie entirely on $1$-edges;
\item
\label{item:regularization.3}
	for each nontrivial contour,
	this contour shall be temporarily filled in with an
	``outer'' face, which shall be regarded as a $2$-face,
	and which shall be removed at the end of the procedure;
\item
\label{item:regularization.4}
	each maximal connected subcomplex without $1$-edges
	(\ie\ containing only $0$-edges and $0$-faces)
	shall be collapsed into a point, unless
	it forms a closed connected component, in which case it shall be
	entirely removed;
	after this operation,
	the star of each vertex looks like
	a finite collection of discs with their centers ``glued''
	to this vertex
	(equivalently, the link is a disjoint union of circles);
\item
\label{item:regularization.5}
	split each vertex star which is ``glued'' from more than one disc
	into a disjoint union of these discs;
	after this operation, each nontrivial connected component
	is again a closed combinatorial surface;
\item
\label{item:regularization.6}
	remove all connected components that consist entirely
	of $1$-faces
	(observe that at this stage all such components are spherical);
\item
\label{item:regularization.7}
	collapse each of the remaining $1$-faces into an edge
	(observe that at this stage all $1$-faces are \emph{digons\/});
\item
\label{item:regularization.8}
	remove previously added ``outer'' faces.
\end{enumerate}

There is an alternative way to define the regularization of a given
augmented diagram $\Delta$ which shows its invariance.
As was noted after the definition of a map,
a map can be ``essentially''
reconstructed from the following data:
its characteristic boundary and
the collection of the characteristic boundaries of its faces
together with the information
which pairs of oriented edges shall be ``identified.''
To obtain the regularization of $\Delta$,
take its characteristic boundary and
the characteristic boundaries of all of its $2$-faces,
collapse all $0$-edges in the characteristic boundary of $\Delta$
(replacing circles consisting entirely of $0$-edges
with trivial components),
and reconstruct a new map by requiring that two oriented
edges be mapped to the same oriented edge if
their images in $\Delta$ either coincided
or were connected by a ``band'' of $1$-faces.
The characteristic contours of $\Delta$ and of
its faces shall be inherited by the regularization of $\Delta$
in the natural sense.

\begin{lemma}
\label{lemma:regularization.(3.3)}
Let\/ $\Delta_0$ be a connected augmented diagram
and let\/ $\Delta$ be its regularization\textup.
Let\/ $m$ be the number of contours of\/ $\Delta$
\textup(the same as for $\Delta_0$\textup)\textup,
and\/ $n$ be the number of connected components
of\/ $\Delta$\textup.
Then
\begin{enumerate}
\item
\label{item:lemma.regularization.1}
	$-m\ge\chi(\Delta)-2n\ge\chi(\Delta_0)-2$,
\item
\label{item:lemma.regularization.2}
	if\/ $\Delta_0$ is orientable\textup,
	then so is\/~$\Delta$\textup,
\item
\label{item:lemma.regularization.3}
	suppose\/ $\chi(\Delta)-2n=\chi(\Delta_0)-2$\textup,
	and suppose that\/ $\Psi$ is a nontrivial connected
	component of\/ $\Delta$ such that all the other
	components either are trivial or have
	spherical closure\/\textup;
	then the closures of\/ $\Psi$ and\/ $\Delta_0$
	are combinatorially equivalent\textup.
\end{enumerate}
\end{lemma}
\begin{proof}[Outline of a proof]
The first inequality of (\ref{item:lemma.regularization.1}) follows from
Corollary \ref{corollary.lemma:maps__positive_Euler_characteristic.(3.2)}.
To prove the second inequality, observe that collapsing
a proper connected subcomplex of a closed connected
combinatorial surface
into a point increases the Euler characteristic by at least $k-1$
if $k$ is the number of disjoint circles in the link of the newly
obtained vertex
(see Corollary \ref{corollary.lemma:maps__positive_Euler_characteristic.(3.2)}),
while the subsequent splitting of the star of this vertex increases
the Euler characteristic by $k-1$ and increases
the number of connected components by at most $k-1$.

It is more or less straightforward to prove
(\ref{item:lemma.regularization.2}).

To prove (\ref{item:lemma.regularization.3}),
observe that if $\chi(\Delta)-2n=\chi(\Delta_0)-2$,
then each operation of collapsing/splitting described above
preserves (up to combinatorial equivalence)
the \emph{connected sum\/} of the closures
of all nontrivial connected components.
(Given a finite set of closed connected combinatorial surfaces, its
\emph{connected sum\/} can be defined up to combinatorial equivalence
similarly to connected sum of topological surfaces.
If $\Psi$ is a connected sum of $\Psi_1,\dotsc,\Psi_k$,
then $\chi(\Psi)-2=\chi(\Psi_1)+\dotsb+\chi(\Psi_k)-2k$.)
\end{proof}

\subsubsection{Diamond move}
\label{subsubsection:maps-diagrams.diagrams.diamond_move}

If $e_1$ and $e_2$ are two oriented edges of an augmented diagram $\Delta$
with a common terminal vertex and identical labels,
then the \emph{diamond move\/}
along $e_1$ and $e_2$ is defined,%
\footnote{Diamond moves in diagrams correspond to \emph{bridge moves\/}
in \emph{pictures}, see \cite{Huebschmann:1981:a2cupW,Rourke:1979:ptg}.}
see \cite[Section 1.4]{CollinsHuebschmann:1982:sdir}
or \cite[Section 3.3.3]{Muranov:2007:fgisgicw}.
There are three kinds of diamond move:
\emph{proper}, \emph{untwisting}, and \emph{disconnecting}.
In the case when none of the oriented edges $e_1$ and $e_2$ is a loop,
all possible diamond moves are shown on Figure~\ref{figure:1}.
\begin{figure}\centering
\includegraphics[scale=1.1]{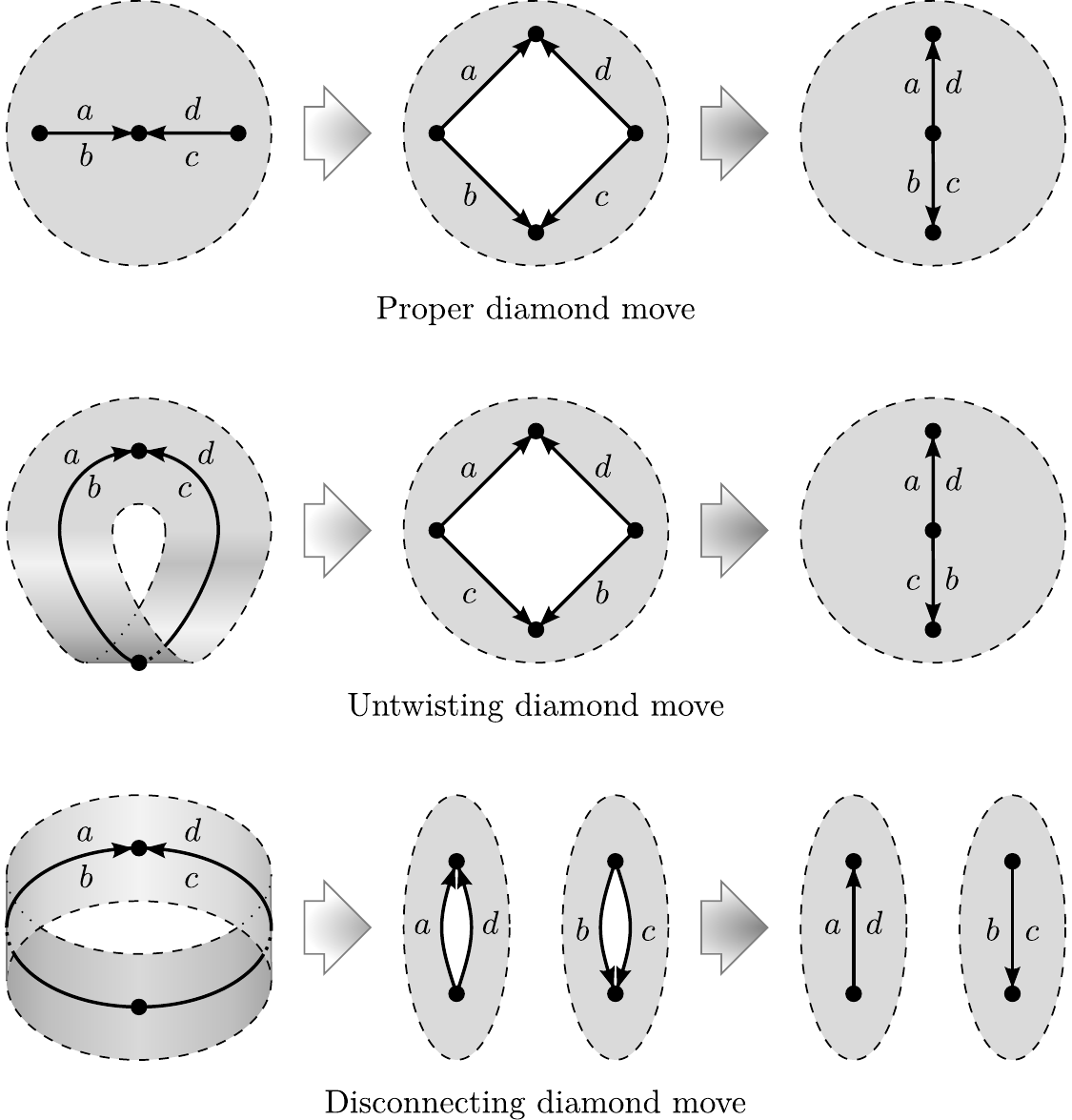}
\caption{Diamond move.}
\label{figure:1}
\end{figure}

In general, the diamond move is defined as follows:
\begin{enumerate}
\item
	let $v$ be the common terminal vertex of $e_1$ and $e_2$,
	fix an orientation locally around $v$, and define
	\emph{left\/} and \emph{right\/} sides of $e_1$ and $e_2$ according to
	how they \emph{enter\/} $v$
	(and not how they \emph{leave\/} $v$, if they are loops);
\item
	cut $\Delta$ along $e_1$ and $e_2$, and glue the right side of $e_1$
	with the left side of $e_2$, and the right side of $e_2$ with
	the left side of $e_1$; label the new $4$ oriented edges naturally
	(with $\ell(e_1)=\ell(e_2)$ and $\ell(e_1^{-1})=\ell(e_2^{-1})$)
\end{enumerate}
(if the component of $\Delta$ that contains $e_1$ and $e_2$ is not closed,
and $e_1$ or $e_2$ is traversed by one of its contours,
then start by temporarily closing that component,
and remove the temporary ``outer'' faces at the end.)

Note that a diamond move does not change the number of edges or the number
of faces, and it does not change the contour labels of faces.
It cannot decrease the number of vertices or increase it
by more than $2$.
Call a diamond move \emph{proper\/} if it does not change the number
of vertices (and in particular preserves the Euler characteristic),
\emph{untwisting\/} if it increases the number of vertices by $1$,
and \emph{disconnecting\/} if it increases the number of vertices by~$2$.

Some useful properties of diamond moves are summarized by the following
lemma.

\begin{lemma}
\label{lemma:diamond_move.(3.4)}
Let\/ $\Delta_0$ be an arbitrary connected diagram
and\/ $\Delta$ be a diagram
obtained from\/ $\Delta_0$ by a diamond move\textup.
Then\/
\begin{enumerate}
\item
	$\Delta$ and\/ $\Delta_0$ have the same number of contours
	and the same contour labels\/\textup;
\item
	$\norm{\Delta_0(0)}\le\norm{\Delta(0)}\le\norm{\Delta_0(0)}+2$\textup,
	$\norm{\Delta(1)}=\norm{\Delta_0(1)}$\textup,
	$\norm{\Delta(2)}=\norm{\Delta_0(2)}$\textup,
	and hence\/ $\chi(\Delta_0)\le\chi(\Delta)\le\chi(\Delta_0)+2$\textup;
\item
	if\/ $\norm{\Delta(0)}=\norm{\Delta_0(0)}$
	\textup(the move is proper\textup)\textup, then\/
	the closures of\/ $\Delta$ and\/ $\Delta_0$ are
	combinatorially equivalent\textup, and the move
	can be\/ \textup{``}undone\/\textup{''}
	by another diamond move\/\textup;
\item
	if\/ $\norm{\Delta(0)}=\norm{\Delta_0(0)}+1$
	\textup(the move is untwisting\textup)\textup, then\/
	$\Delta$ is connected and\/ $\Delta_0$ is non-orientable\/\textup;
\item
	if\/ $\norm{\Delta(0)}=\norm{\Delta_0(0)}+2$
	\textup(the move is disconnecting\textup)\textup, then
	\begin{enumerate}
	\item
		$\Delta$ has has at most\/ $2$ connected components\textup,
	\item
		if\/ $\Delta_0$ is orientable\textup, then
		so is\/~$\Delta$\textup,
	\item
		if\/ $\Delta$ has\/ $2$ connected components
		and\/ $\Delta_0$ is non-orientable\textup, then so is
		\textup(at least one of the components
		of\/\textup)~$\Delta$\textup,
	\item
		if\/ $\Delta$ has\/ $2$ connected components
		and the closure of one of them is spherical\textup, then
		the closure of the other is combinatorially equivalent
		to the closure of\/~$\Delta_0$\textup.
	\end{enumerate}
\end{enumerate}
\end{lemma}
\begin{proof}[Outline of a proof]
Let $e_1$ and $e_2$ be the oriented edges such that $\Delta$
is obtained from $\Delta_0$ by the diamond move along $e_1$ and $e_2$.
The most difficult case is the one when both $e_1$ and $e_2$ are loops.
(If $e_1$ is a loop and $e_2$ is not, then the move is proper.)
Only this case shall be considered here.

Let $\Theta$ be the diagram obtained from $\Delta_0$ by cutting it
along the loops $e_1$ and $e_2$, so that the common terminal vertex
of $e_1$ and $e_2$ in $\Delta_0$ be split into $4$ vertices in $\Theta$.
Then, compared to $\Delta_0$,
$\Theta$ has $1$, $2$, or $3$ additional contours,
and the sum of the lengths of these additional contours is $4$.
Let $a$, $b$, $c$, $d$ be the oriented edges of $\Theta$ corresponding
to the left and right sides of $e_1$ and left and right sides of $e_2$,
respectively; they all lie on the additional contours of $\Theta$.
Then $\Delta$ is obtained from $\Theta$ by gluing
$b$ with $c$ and $d$ with~$a$.

Without loss of generality, it is enough to consider the
following $3$ cases and their $9$ subcases (see Figure \ref{figure:2}):
\begin{figure}\centering
\includegraphics[scale=1.1]{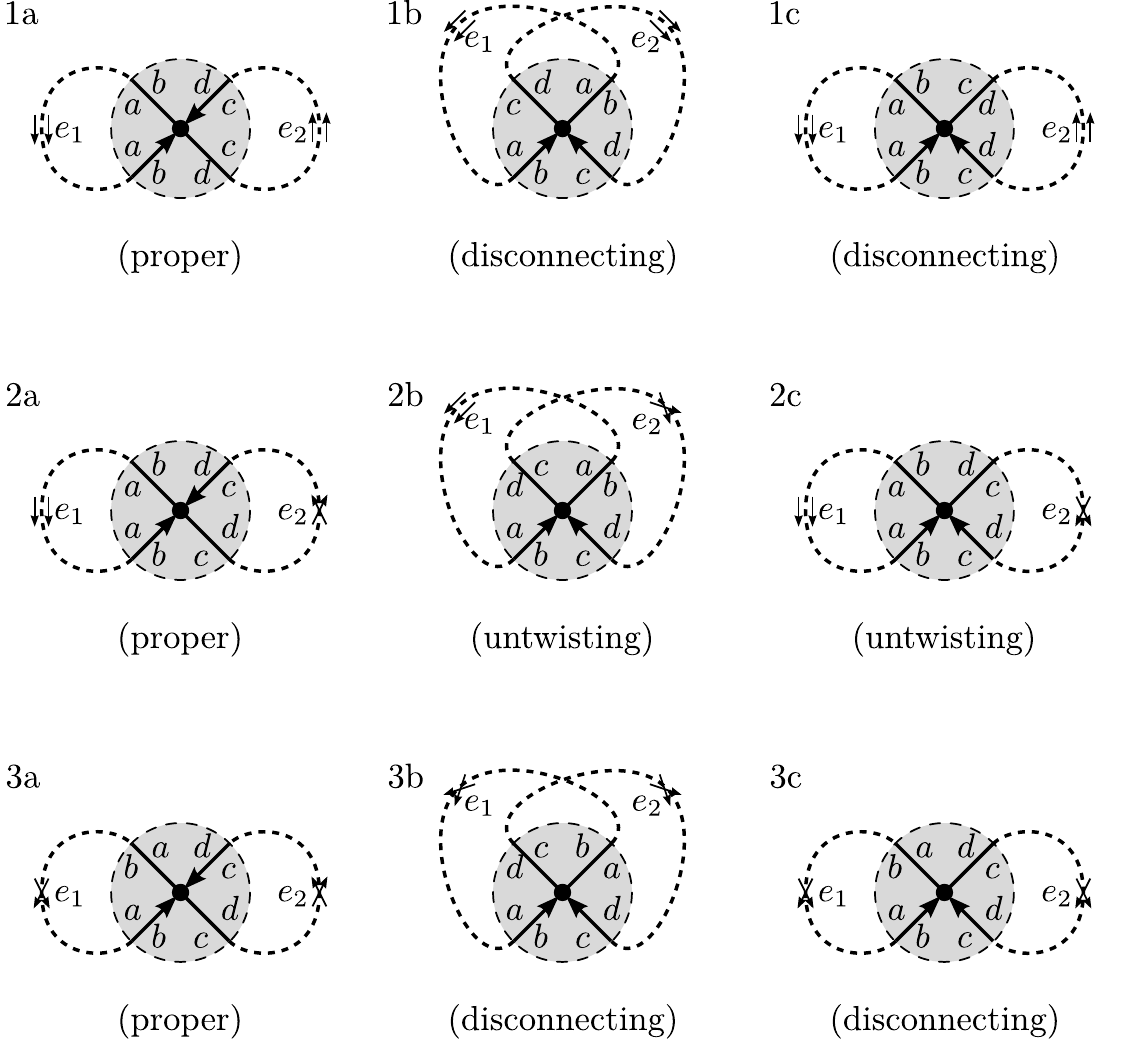}
\caption{Oriented loops $e_1$ and $e_2$.}
\label{figure:2}
\end{figure}
\begin{enumerate}
\item
	the orientation is preserved along both $e_1$ and $e_2$,
	and hence
	\begin{enumerate}
	\item
		either $\Theta$ has $3$ additional contours $a,bd,c$,
		and the move is proper, or
	\item
		$\Theta$ has $1$ additional contour $acb^{-1}d^{-1}$,
		and the move is disconnecting
		but cannot change the number of connected components, or
	\item
		$\Theta$ has $3$ additional contours $a,bc^{-1},d$,
		and the move is disconnecting;
	\end{enumerate}
\item
	the orientation is preserved along $e_1$ and inverted along $e_2$,
	and hence
	\begin{enumerate}
	\item
		either $\Theta$ has $2$ additional contours $a,bcd$,
		and the move is proper, or
	\item
		$\Theta$ has $1$ additional contour $adbc^{-1}$,
		and the move is untwisting, or
	\item
		$\Theta$ has $2$ additional contours $a,bc^{-1}d^{-1}$,
		and the move is untwisting;
	\end{enumerate}
\item
	the orientation is inverted along both $e_1$ and $e_2$,
	and hence
	\begin{enumerate}
	\item
		either $\Theta$ has $1$ additional contour $abcd$,
		and the move is proper, or
	\item
		$\Theta$ has $2$ additional contours $ad,bc^{-1}$,
		and the move is disconnecting, or
	\item
		$\Theta$ has $1$ additional contour $abc^{-1}d^{-1}$,
		and the move is disconnecting
		but cannot change the number of connected components.
	\end{enumerate}
\end{enumerate}

Details are left to the reader.
\end{proof}

\subsection{Reduced diagrams}
\label{subsection:maps-diagrams.reduced_diagrams}

A pair of distinct faces $\{\Pi_1,\Pi_2\}$ in a diagram
over \tGP{X}{R} shall be called
\emph{immediately strictly cancelable\/} if
the labels of $\Pi_1$ and $\Pi_2$ are identical as indexed members of $R$,
both faces are incident to some common edge, and
their contours
can be written as $\cntr{\Pi_1}=p_1sq_1$ and $\cntr{\Pi_2}=p_2sq_2$
with a common nontrivial subpath $s$, $\abs{s}>0$,
so that $\abs{p_1}=\abs{p_2}$ (and $\abs{q_1}=\abs{q_2}$).
A diagram shall be called
\emph{weakly strictly reduced\/} if
it does not have immediately strictly cancelable pairs of faces.

A pair of distinct faces $\{\Pi_1,\Pi_2\}$ in a diagram
shall be called \emph{diagrammatically cancelable\/} if
there exists a sequence of diamond moves that separates
these two faces into a spherical subdiagram
(\ie\ leads to a diagram in which the faces corresponding to
$\Pi_1$ and $\Pi_2$ form a spherical connected component).
A diagram shall be called
\emph{diagrammatically reduced\/}
if it does not have diagrammatically cancelable pairs of faces.

Clearly, an immediately strictly cancelable pair is
diagrammatically cancelable,
and hence a diagrammatically reduced diagram is weakly strictly reduced.%
\footnote{There exist different non-equivalent definitions of
cancelable pairs, reduced diagrams, etc.
The modifiers such as ``weakly,'' ``strictly,''
``diagrammatically'' are added here to distinguish among them.
For example, diagrams reduced in the sense of \cite{LyndonSchupp:2001:cgt}
would be called \emph{weakly diagrammatically reduced\/} here.}

A diagram shall be called
\emph{diamond-move reduced\/}
if no sequence of diamond moves can increase its Euler characteristic.

By Lemma \ref{lemma:diamond_move.(3.4)}, every diagram can be transformed
into a diamond-move reduced one by a sequence of diamond moves,
and if a diagram is diamond-move reduced, then any sequence of diamond
moves preserves its closure up to combinatorial equivalence
(and in particular it preserves its Euler characteristic and
orientability).

A diamond-move reduced diagram is diagrammatically reduced
if and only if does not have spherical components with exactly $2$ faces.

\begin{lemma}
\label{lemma:diagram_reduction.(3.5)}
Let\/ $\Delta_0$ be an augmented diagram over\/ \tGP{X}{R}\textup.
Let\/ $\Delta_1$ be the regularization of\/ $\Delta_0$\textup,
and $\Delta$ be a diamond-move reduced diagram obtained from\/ $\Delta_1$
by some sequence of diamond moves\textup.
If\/ $\Delta$ is not closed\textup, then let $\Delta'$
be the diagram obtained from\/ $\Delta$ by discarding
all closed connected components\textup.
Then
\begin{enumerate}
\item
\label{item:lemma.diagram_reduction.1}
	$\Delta$ has the same number of contours as\/ $\Delta_0$\textup,
	and contour labels of\/ $\Delta$ are obtained from
	the corresponding contour labels of\/ $\Delta_0$
	by canceling the\/ $1$ symbols\/\textup;
\item
\label{item:lemma.diagram_reduction.2}
	each component of\/ $\Delta$ either is diagrammatically reduced
	or is a spherical diagram with\/ $2$ faces\/\textup;
\item
\label{item:lemma.diagram_reduction.3}
	if\/ $\Delta_0$ is orientable\textup,
	then so is\/ $\Delta$\textup;
	moreover\textup,
	any orientation of\/ $\Delta_0$
	naturally induces an orientation of\/ $\Delta$ so that
	positive contours which do not become trivial
	stay positive\textup,
	and negative which do not become trivial
	stay negative\/\textup;
\item
\label{item:lemma.diagram_reduction.4}
	if the closures of all nontrivial connected components of\/ $\Delta_0$
	are spherical\textup,
	then so are the closures of all nontrivial components
	of\/~$\Delta$\textup;
\item
\label{item:lemma.diagram_reduction.5}
	if\/ $\Delta_0$ is disc\textup,
	then\/ $\Delta'$ is disc\/\textup;
\item
\label{item:lemma.diagram_reduction.6}
	if\/ $\Delta_0$ is annular\textup,
	then\/ $\Delta'$ either is annular or consists of\/ $2$
	disc components\/\textup;
\item
\label{item:lemma.diagram_reduction.7}
	if\/ $\Delta_0$ is a connected one-contour diagram\textup,
	then\/ $\chi(\Delta')\ge\chi(\Delta_0)$\textup,
	and if additionally\/ $\chi(\Delta')=\chi(\Delta_0)$\textup,
	then the closures of\/ $\Delta'$ and\/ $\Delta_0$
	are combinatorially equivalent\textup,
	and all closed components of\/ $\Delta$
	are spherical\textup.
\end{enumerate}
\end{lemma}
No proof shall be given as this lemma
follows more or less directly from
Corollary \ref{corollary.lemma:maps__positive_Euler_characteristic.(3.2)}
and
Lemmas \ref{lemma:regularization.(3.3)}, \ref{lemma:diamond_move.(3.4)}.

One of basic results of combinatorial group theory is the fact that
a group word $w$ represents the trivial element in the group presented
by \tGP{X}{R} if and only if there exists a disc diagram over \tGP{X}{R}
with contour label $w$.
The following lemma summarizes several results of this kind
that shall be used in the proof of the main theorem.

\begin{lemma}
\label{lemma:diagrams.(3.6)}
Let\/ $G$ be the group presented by\/ \tGP{X}{R}\textup.
Let\/ $w$\textup, $w_1$\textup, and\/ $w_2$ be arbitrary group words
over\/ $X$\textup, and\/ $n$ a positive integer\textup.
Then
\begin{enumerate}
\item
\label{item:lemma.diagrams.1}
	if there exists an
	augmented disc diagram over\/ \tGP{X}{R}
	with contour label\/ $w$\textup, then\/ $[w]=1$
	in\/~$G$\textup;
\item
\label{item:lemma.diagrams.2}
	if\/ $[w]=1$ in\/ $G$\textup, then there exists a diamond-move reduced
	disc diagram over\/ \tGP{X}{R}
	with contour label\/~$w$\textup;
\item
\label{item:lemma.diagrams.3}
	if there exists an
	augmented oriented annular diagram over\/ \tGP{X}{R}
	whose contours agree with the orientation and whose
	contour labels are\/ $w_1$ and\/ $w_2^{-1}$\textup, then\/
	$[w_1]$ and\/ $[w_2]$ are conjugate in\/~$G$\textup;
\item
\label{item:lemma.diagrams.4}
	if\/ $[w_1]$ and\/ $[w_2]$ are conjugate in\/ $G$
	and\/ $[w_1]\ne1$\textup, then there exists an oriented
	diamond-move reduced annular diagram over\/ \tGP{X}{R}
	whose contours agree with the orientation and whose
	contour labels are\/ $w_1$ and\/~$w_2^{-1}$\textup;
\item
\label{item:lemma.diagrams.5}
	if there exists an
	augmented one-contour diagram over\/ \tGP{X}{R}
	with contour label\/ $w$ and whose closure
	is a combinatorial sphere with\/
	$n$ handles\textup,
	then\/
	$[w]\in[G,G]$ and\/ $\cl_G([w])\le n$\textup,
	$\sql_G([w])\le 2n+1$\textup;
\item
\label{item:lemma.diagrams.6}
	if\/ $\cl_G([w])=n>0$\textup, then there exists a
	diamond-move reduced one-contour diagram over\/ \tGP{X}{R}
	with contour label\/ $w$ and whose closure
	is a combinatorial sphere with\/ $n$ handles\/\textup;
\item
\label{item:lemma.diagrams.7}
	if there exists an
	augmented connected one-contour diagram over\/ \tGP{X}{R}
	with contour label\/ $w$ and whose closure
	is a non-orientable combinatorial surface of
	Euler characteristic\/ $2-n$\textup,
	then\/ $[w]\in G^2$ and\/ $\sql_G([w])\le n$\textup;
\item
\label{item:lemma.diagrams.8}
	if\/ $\sql_G([w])=n>0$\textup,
	then there exists a
	diamond-move reduced connected one-contour diagram over\/ \tGP{X}{R}
	with contour label\/ $w$ and whose closure
	is either a non-orientable combinatorial surface of
	Euler characteristic\/ $2-n$ or
	an orientable combinatorial surface of
	Euler characteristic\/ $3-n$
	\textup(\ie\ a combinatorial sphere with\/
	$(n-1)/2$ handles\/\textup)\textup.
\end{enumerate}
\end{lemma}
\begin{proof}[Outline of a proof]
Most of these statements are usually considered well-know.
Statements analogous or equivalent to
(\ref{item:lemma.diagrams.1}),
(\ref{item:lemma.diagrams.2}),
(\ref{item:lemma.diagrams.3}),
(\ref{item:lemma.diagrams.4}) may be found, for example, in
Theorem V.1.1 and Lemmas V.1.2, V.5.1, V.5.2 in
\cite[Chapter V]{LyndonSchupp:2001:cgt} and in
Lemmas 11.1, 11.2 in
\cite[Chapter 4]{Olshanskii:1989:gosg-rus,Olshanskii:1991:gdrg-eng}.
A care should be taken, however, since definitions of diagrams
are not exactly equivalent everywhere, and because in the latter reference
all results are stated only for augmented diagrams.
Here only outlines of proofs of
(\ref{item:lemma.diagrams.5}),
(\ref{item:lemma.diagrams.7}),
(\ref{item:lemma.diagrams.8})
shall be given.
Statement
(\ref{item:lemma.diagrams.6})
can be proved similarly to (\ref{item:lemma.diagrams.8}).

To prove (\ref{item:lemma.diagrams.5}),
it suffices to notice that if $\Delta$ is
a one-contour map whose closure is a combinatorial
sphere with $n$ handles,
then $\cntr{}\Delta$ is the product of $n$ commutators
in the fundamental group of $\Delta$,
and that every product of $n$ commutators is also the product of
$2n+1$ squares, see \cite[Section 2.5]{Culler:1981:ussefg}.
Similarly, to prove (\ref{item:lemma.diagrams.7}),
it suffices to notice that if $\Delta$ is
a one-contour map whose closure is a non-orientable combinatorial surface
of Euler characteristic $2-n$,
then $\cntr{}\Delta$ is the product of $n$ squares
in the fundamental group of~$\Delta$.

Here follows an outline of a proof of
(\ref{item:lemma.diagrams.8}).

Assume $\sql_G([w])=n>0$.
Then there exists an augmented \emph{simple\/} disc diagram $\Phi$
over \tGP{X}{R} with the contour label of the form
$v_1^2v_2^2\dotsm v_n^2w^{-1}$.
After gluing together accordingly parts of the contour corresponding
to distinct occurrences of each $v_i$, obtain
an augmented connected one-contour diagram $\Delta_0$ over \tGP{X}{R}
with contour label $w$ and whose
closure is a non-orientable combinatorial surface of
Euler characteristic $2-n$ (and $\chi(\Delta_0)=1-n$).
Consider a diamond-move reduced diagram obtained from
the regularization of $\Delta_0$
by some sequence of diamond moves, and let
$\Delta$ be its non-closed connected component
(\ie\ the component containing its contour).
Then, by Lemma \ref{lemma:diagram_reduction.(3.5)},
$\ell(\cntr\Delta)=\ell(\cntr\Delta_0)=w$,
and either the closure of $\Delta$ is combinatorially equivalent
to $\Delta_0$, in which case $\Delta$ is a desired diagram,
or $\chi(\Delta)\ge\chi(\Delta_0)+1=2-n$.
It is left to show that in the second case $\Delta$ is orientable
and $\chi(\Delta)=2-n$.
Indeed, if $\chi(\Delta)\ge 2-n$ and $\Delta$ is non-orientable,
then $\sql_G([w])\le n-1$ by (\ref{item:lemma.diagrams.7}).
If $\Delta$ is orientable
and $\chi(\Delta)\ge 3-n$, then
$\sql_G([w])\le n-1$ again by (\ref{item:lemma.diagrams.5}).
Thus (\ref{item:lemma.diagrams.8}) is proved.
\end{proof}

\begin{corollary.lemma}
\label{corollary.lemma:diagrams.(3.7)}
Let\/ $G$ be the group presented by\/ \tGP{X}{R}
and\/ $w$ a group word over\/ $X$\textup.
Then
\begin{enumerate}
\item
\label{item:corollary.lemma.diagrams.1}
	$[w]\in G^2$ if and only if there exists a
	one-contour diagram over\/ \tGP{X}{R} with contour label\/~$w$\textup;
\item
\label{item:corollary.lemma.diagrams.2}
	$[w]\in[G,G]$ if and only if there exists an orientable
	one-contour diagram over\/ \tGP{X}{R}
	with contour label\/~$w$\textup.
\end{enumerate}
\end{corollary.lemma}

\subsection{Asphericity}
\label{subsection:maps-diagrams.asphericity}

There exist different non-equivalent definitions of \emph{asphericity\/}
of a group presentation in the literature.
Group presentations themselves may also be viewed differently,
notable variations including the following:
whether or not the relators are required to be reduced,
or even cyclically reduced, and
whether or not the relators form a set or an indexed family.
Different notions of a group presentation necessitate different
definitions of asphericity.
In this paper, relators are allowed to be non-reduced, and even
freely trivial as long as they are not trivial (\ie\ not of length $0$),
and they may form an indexed family rather than just a set.
Consistent definitions of different kinds of asphericity
of a group presentation shall be borrowed from
\cite{ChiswellColHue:1981:agp,CollinsHuebschmann:1982:sdir}.
Only \emph{diagrammatic\/} and \emph{singular\/} asphericities
shall be used in this paper.

A group presentation is
\emph{diagrammatically aspherical\/} if
there exist no diagrammatically reduced spherical diagrams over it
(see \cite{CollinsHuebschmann:1982:sdir}).
Equivalently, a group presentation is diagrammatically aspherical
if and only if
every spherical diagram over this presentation can be transformed
by a sequence of diamond moves into a diagram whose
all connected components
(which are necessarily spherical by Lemma \ref{lemma:diamond_move.(3.4)})
have exactly $2$ faces.

Observe that a diagrammatically aspherical presentation cannot
have freely trivial relators because otherwise
there would exist one-face spherical diagrams over it.
A group presentation \tGP{X}{R} is
\emph{singularly aspherical\/} if
it is diagrammatically aspherical,
no element of $R$ represents a proper power in
(the free group presented by)
\tGP{X}{\varnothing},
and no two distinct indexed members of $R$
are conjugate or conjugate to each other's inverses
in~\tGP{X}{\varnothing}.

A group shall be called \emph{diagrammatically\/} or
\emph{singularly aspherical\/}, accordingly,
if it has a presentation which is such.

It is considered well-known that
``aspherical in a certain sense groups
are of cohomological and geometric
dimension at most $2$, and hence are torsion-free.''
The following lemma states this result precisely using
the chosen terminology.

\begin{lemma}
\label{lemma:singularly_aspherical_groups.(3.8)}
Every nontrivial singularly aspherical group either is free\/
\textup(hence of cohomological and geometric dimension\/ $1$\textup)
or has cohomological and geometric dimension\/ $2$\textup;
in both cases it is torsion-free\textup.
\end{lemma}
\begin{proof}[Outline of a proof]
Let $G$ be an arbitrary
nontrivial non-free
singularly aspherical group,
and let \tGP{X}{R} be its singularly aspherical presentation.
The goal is to show that in this case $K(X;R)$ is an
\emph{Eilenberg-MacLane complex\/} of type $(G,1)$,
which is equivalent to showing that $H_2(\tilde K(X;R))=0$,
where $\tilde K(X;R)$ is the universal cover of $K(X;R)$.

Every $2$-cycle in the group of $2$-dimensional cellular chains
of $\tilde K(X;R)$ can be represented by
a morphism to $\tilde K(X;R)$ of some (generally not connected)
closed oriented combinatorial surface.
Such a surface is naturally a closed diagram over \tGP{X}{R},
the structure of a diagram is induced by the morphism.
Since $\tilde K(X;R)$ is simply connected, every such morphism
of a closed oriented diagram can be transformed without changing
the represented homology class to
a morphism of an oriented diagram whose all connected components are
spherical
(here it may be helpful to pass to augmented diagrams first
and then to use Lemma \ref{lemma:regularization.(3.3)}).
Because applying a diamond move to the diagram and changing accordingly
the morphism does not change the represented homology class, and
because the presentation is diagrammatically aspherical,
the morphism and the diagram can be transformed by diamond
moves to a morphism of an oriented
diagram whose all connected components are
spherical with $2$ faces each.
Since the presentation is singularly aspherical,
each morphism of an oriented spherical diagram with $2$
faces represents $0$ in the group of $2$-dimensional chains
(because the images of the two faces coincide).
Hence every cellular $2$-cycle in $\tilde K(X;R)$
represents the trivial homology class, and thus
$H_2(\tilde K(X;R))=0$.

Since $K(X;R)$ is a $2$-dimensional
Eilenberg-MacLane complex of type $(G,1)$, the
geometric and cohomological dimensions of $G$ are at most $2$.
The cohomological dimension of $G$ cannot be $1$ because
by Stallings-Swan theorem \cite{Stallings:1968:tfgime,Swan:1969:gcd1},
every group of cohomological dimension $1$ is free.
Thus both the cohomological and the
geometric dimensions of $G$ are $2$
(the cohomological dimension is always less then or equal to
the geometric dimension).
Corollary VIII.2.5 in \cite[Chapter VIII]{Brown:1994:cg} states that
every group of finite cohomological dimension is torsion-free.
\end{proof}

It is worth mentioning here that there is also a notion of
\emph{combinatorial asphericity},
which is weaker than
diagrammatic asphericity: a presentation
is \emph{combinatorial aspherical\/} if its
\emph{Cayley complex\/}
(which is obtained from $\tilde K(X;R)$ by identifying certain faces)
is \emph{aspherical},
though in this form the definition only makes good sense
for presentations in which all relators are cyclically reduced
and hence the Cayley complex is well-defined.
Relators of a combinatorially aspherical presentation may represent
proper powers in the free group, and a
combinatorially aspherical group may have torsion,
but Theorem 3 in \cite{Huebschmann:1979:ctagscg}
shows that finite-order elements in such a group
are only the ``obvious'' ones.



\section{Technical lemmas}
\label{section:technical_lemmas}

The idea of the proof of the main theorem is to use
properties of the constructed presentation \tGP{a,b}{\mathcal R}
to show nonexistence of certain diagrams with
certain contour labels.
The desired nonexistence can be proved by contradiction,
by using the imposed small-cancellation-type conditions
and some combinatorics to show that
all edges of a hypothetical diagram $\Delta$
can be distributed among its faces in such a way
that sufficiently few edges be assigned to each face,
and hence coming to a contradiction with the equality
$$
\norm{\Delta(1)}
=\sum_{\Pi\in\Delta(2)}\!\frac{1}{2}\abs{\cntr\Pi}
+\sum_{i}\frac{1}{2}\abs{\cntr_i\Delta}.
$$

For example, to prove that the group $G$ is nontrivial, it is enough
to show that there exist no disc diagram over \tGP{a,b}{\mathcal R}
with the contour of length $1$, and hence $[a]_G\ne1$.
This can be done as follows.
If there exists a disc diagram over \tGP{a,b}{\mathcal R}
with the contour of length $1$, then
it can be transformed into another such
diagram which will be ``convenient'' in a certain sense,
and all the edges of this new diagram
can be distributed among its faces in such
a way that the number of edges assigned to each face be strictly
less than half of the contour length of that face,
but this is impossible because of the aforementioned equality.

Lemmas \ref{lemma:estimating_lemma.selected_arcs.(4.3)},
\ref{lemma:estimating_lemma.exceptional_arcs.(4.4)},
and \ref{lemma:inductive_lemma.(4.5)}
(Estimating Lemmas and Inductive Lemma)
proved in this section
shall be the main tools in proving by induction
that edges of the diagrams under consideration
can be distributed
among their faces as needed.

Some of the definitions and results of this section
are simplified versions
of the corresponding definitions and results
of \cite{Muranov:2007:fgisgicw} and may be not
exactly equivalent.

\begin{definition}
An \emph{$S_1$-map\/} is a map together with a system of nontrivial
reduced paths in the characteristic boundaries of its faces,
called \emph{selected paths},
satisfying the following conditions:
\begin{enumerate}
\item
	the inverse path of every selected path is selected,
\item
	every nontrivial subpath of every selected path is selected,
\item
	the image of every selected path in the $1$-skeleton of the map
	is reduced, and
\item
	distinct maximal selected paths do not overlap
	(meaning that no edge lies on both of them)
	unless they are mutually inverse,
	and maximal selected paths are simple
	(make at most one full circle).
\end{enumerate}
\end{definition}

(The last condition of this definition is not present in
\cite{Muranov:2007:fgisgicw};
it is added here to slightly simplify the definition
of $\kappa$ and $\kappa'$ and the proof of
Lemma~\ref{lemma:estimating_lemma.selected_arcs.(4.3)}.)

An internal arc of an $S_1$-map is called \emph{selected\/} if
it lies entirely on the images of selected paths
``from both sides.''

\begin{definition}
Let $\Delta$ be an $S_1$-map, $\Phi$ its simple disc submap.
Then $\Delta$ is said to satisfy
the \emph{condition\/ $\mathsf Z(2)$ relative to\/ $\Phi$} if
no cyclic shift of $\cntr{}\Phi$ can be decomposed
as the concatenation of $2$ paths each of which either is trivial
or is the image of a selected path in the characteristic boundary
of a face that does not belong to $\Phi$ (is outside of $\Phi$).
\end{definition}

Figure \ref{figure:3} shows two situations prohibited by the condition
$\mathsf Z(2)$ relative to $\Phi$,
where $q,q_1,q_2$ are selected paths in the characteristic
boundaries of $\Pi,\Pi_1,\Pi_2$, respectively.
\begin{figure}\centering
\includegraphics[scale=1.1]{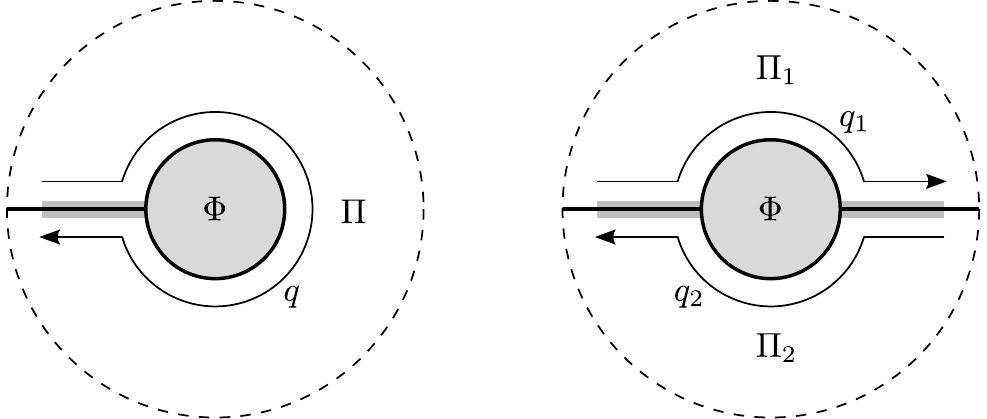}
\caption{Two situations prohibited by the condition $\mathsf Z(2)$.}
\label{figure:3}
\end{figure}

\begin{definition}
An \emph{indexed\/} map is a map $\Delta$ together with a function
$\iota\colon \Delta(2)\to I$, where $I$ is an arbitrary index set.
The value $\iota(\Pi)$ is called the \emph{index\/} of the face~$\Pi$.
\end{definition}

\begin{definition}
An \emph{$S_2$-map\/} is an indexed map together with
a system of non-overlapping internal arcs
called \emph{exceptional\/} arcs, such that if two faces are incident
to the same exceptional arc, they have the same index.
Exceptional arcs shall be assigned indices according
to the indices of the incident faces.
\end{definition}

If $\Gamma$ is a submap of an $S_2$-map $\Delta$,
then $\Gamma$ naturally inherits the structure of an $S_2$-map:
an arc of $\Gamma$ is exceptional if and only if it is exceptional
in $\Delta$ and internal in $\Gamma$
(this is different from~\cite{Muranov:2007:fgisgicw}).

\begin{definition}
Let $\Gamma$ be a connected submap of an $S_2$-map $\Delta$,
and for every $i$, let $A_i$ denote the set of all index-$i$
exceptional arcs of $\Gamma$
(\ie\ exceptional arcs of $\Delta$ that are internal in $\Gamma$),
$B_i$ denote the set of all index-$i$ faces of $\Gamma$,
and $K_i$ denote the set of all connected
components of the subcomplex of $\Gamma$ obtained by
removing all the faces that are in $B_i$ and all the arcs
that are in $A_i$.
Then $\Gamma$ is said to satisfy
the \emph{condition\/ $\mathsf Y$ relative to\/ $\Delta$} if
for every $i$ such that $A_i\ne\varnothing$,
the number of elements of $K_i$ that either have
Euler characteristic $1$ or contain an index-$i$ exceptional
arc of $\Delta$ incident to a face of $\Gamma$
(which has to be external in $\Gamma$)
is less than or equal to~$\norm{B_i}$.
\end{definition}

Note that in the last definition, the ambient map $\Delta$
only serves to distinguish certain external arcs of $\Gamma$
as \emph{exceptional in\/~$\Delta$}.

\begin{definition}
An $S$-map is a map together with structures of both an $S_1$-map
and an $S_2$-map such that every exceptional arc is selected.
\end{definition}

\begin{definition}
Let $\Gamma$ be a submap of an $S$-map $\Delta$, and let
$\lambda,\mu,\nu$ be functions $\Gamma(2)\to[0,1]$.
Then $\Gamma$ is said to satisfy the
\emph{condition\/ $\mathsf D(\lambda,\mu,\nu)$
relative to\/ $\Delta$} if
the following three conditions hold:
\begin{description}
\item[$\mathsf D_1(\lambda)$]
	if $\Pi$ is a face of $\Gamma$ and
	$L$ is the number
	of edges of the characteristic boundary of $\Pi$
	that do not lie on any selected path,
	then
	$$
	L\le\lambda(\Pi)\abs{\cntr\Pi};
	$$
\item[$\mathsf D_2(\mu)$]
	if $\Pi$ is a face of $\Gamma$,
	$u$ is a selected (internal) arc of $\Delta$ incident to $\Pi$,
	and $M$ is the number of those edges of $u$
	that do not lie on any exceptional arc of $\Delta$, then
	$$
	M\le\mu(\Pi)\abs{\cntr\Pi};
	$$
\item[$\mathsf D_3(\nu)$]
	if $\Pi$ is a face of $\Delta$,
	$p$ is a simple path in $\Gamma$ which is the image of a selected
	path in the characteristic boundary of $\Pi$, and
	$N$ is the sum of the lengths of all the
	exceptional arcs of $\Delta$ that lie on $p$,
	then
	$$
	N\le\nu(\Theta)\abs{\cntr\Theta}
	$$
	for every face $\Theta$ of $\Gamma$ satisfying
	$\iota(\Theta)=\iota(\Pi)$.
\end{description}
\end{definition}

In the context of proving the main theorem,
diagrams over \tGP{a,b}{\mathcal R}
shall be given the structure of $S$-maps
as shown on Figure \ref{figure:4}.
Namely, the maximal selected paths in the characteristic boundaries
of faces in those diagrams shall
correspond to the ``$u$-subwords'' of the defining relations,
the index of a face shall be the index of its label
(if $\ell(\cntr{}\Pi)=r_n$, then $\iota(\Pi)=n$), and
an internal arc shall be exceptional if and only if it corresponds
``on both sides'' to the the same part of the same ``$u$-subword''
(this can happen even in a weakly strictly reduced diagram since
relators contain pairs of mutually inverse ``$u$-subwords'').
In fact, not all diagrams over \tGP{a,b}{\mathcal R}
shall be used, but only
weakly strictly reduced ones
in which each exceptional arc corresponds to an \emph{entire\/}
``$u$-subword''; such diagrams shall be called ``\emph{convenient}.''
This restriction shall be needed in order for the defined above
condition $\mathsf Y$ to be satisfied.
Note also that if $s$ is a selected
internal arc incident to a face $\Pi$ in any such
$S$-diagram over \tGP{a,b}{\mathcal R},
then $\abs{s}\le\nu_{\iota(\Pi)}\abs{\cntr{}\Pi}$,
and if additionally $s$ is not exceptional, then
$\abs{s}\le\mu_{\iota(\Pi)}\abs{\cntr{}\Pi}$,
and hence such $S$-diagrams satisfy
the condition $\mathsf D$ with appropriate parameters.
\begin{figure}\centering
\includegraphics[scale=1.1]{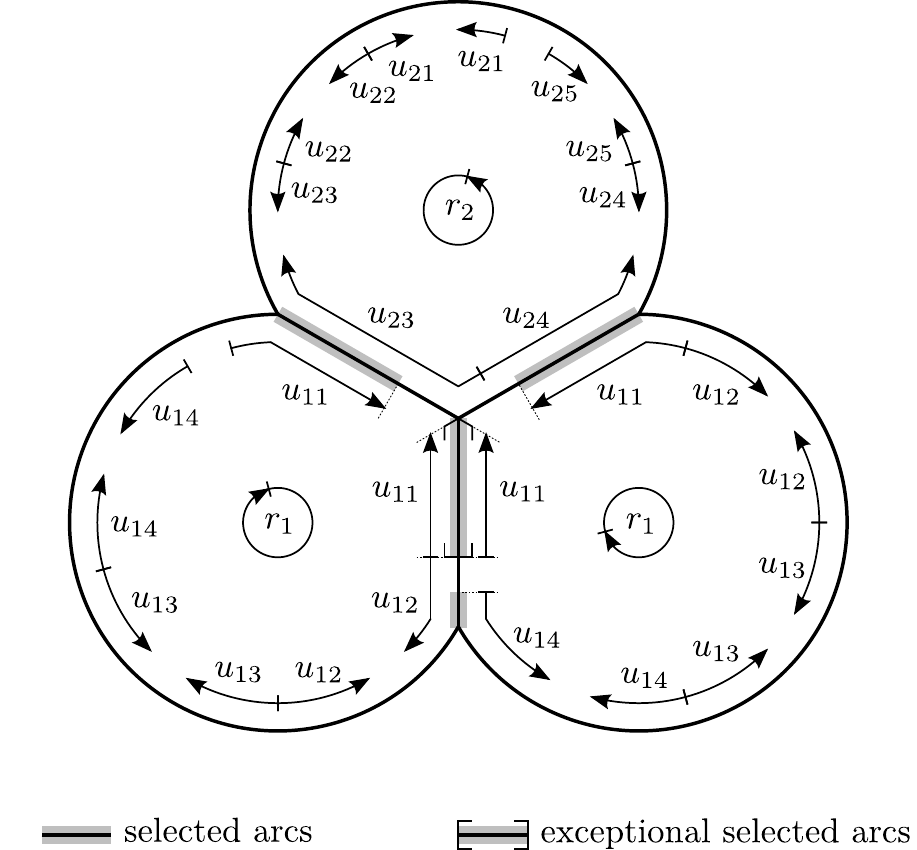}
\caption{A diagram over \tGP{a,b}{r_1,r_2,\dotsc} as an $S$-map.}
\label{figure:4}
\end{figure}

The following notation shall be used starting from
Lemma \ref{lemma:estimating_lemma.selected_arcs.(4.3)}
(First Estimating Lemma).
Let $\Pi$ be a face of an $S_1$-map $\Delta$.
If $\Pi$ has at least one selected path in its characteristic boundary,
then let $\kappa(\Pi)$, or $\kappa_\Delta(\Pi)$,
denote the number of connected components obtained from
the characteristic boundary of $\Pi$
by removing all edges and all intermediate vertices of all its
selected paths,
and let $\kappa'(\Pi)$, or $\kappa'_\Delta(\Pi)$, denote the number of
those components that either do not consist of
a single vertex or consist of a single vertex whose image
in the $1$-skeleton of $\Delta$ has degree~$1$.
If $\Pi$ has no selected paths, then let
$\kappa(\Pi)=\kappa'(\Pi)=0$.
Observe that $\kappa'(\Pi)\le\kappa(\Pi)$ and that $\kappa(\Pi)$
is also the number of maximal selected paths
in the characteristic boundary of $\Pi$ going in the same direction.
Note that if all nontrivial reduced paths in the characteristic boundary
of $\Pi$ are selected, then $\kappa(\Pi)=\kappa'(\Pi)=0$.
In the context of proving the main theorem,
$\kappa(\Pi)=\kappa_{\iota(\Pi)}$.

Lemmas \ref{lemma:estimating_lemma.selected_arcs.(4.3)},
\ref{lemma:estimating_lemma.exceptional_arcs.(4.4)},
\ref{lemma:inductive_lemma.(4.5)} are identical or slightly
simplified versions of Lemmas 50, 54, 58 in \cite{Muranov:2007:fgisgicw};
however, somewhat less formal proofs shall be given below for convenience.

The proof of Lemma \ref{lemma:estimating_lemma.selected_arcs.(4.3)}
relies on the lemma of Philip Hall stated below.

If $R$ is a binary relation and $X$ is a set, then denote
$$
R(X)=\{\,y\mid(\exists x\in X)(x\mathrel Ry)\,\}.
$$

\begin{lemma}[Philip Hall \cite{Hall:1935:ors}]
\label{lemma:Halls_lemma.(4.1)}
Let\/ $A$ and\/ $B$ be two finite sets\textup,
and\/ $R$ be a relation from\/ $A$ to\/ $B$
\textup(\ie\ $R\subset A\times B$\textup{).}
Then the following are equivalent\textup.
\begin{itemize}
\item[\textup{(I)}]
	There exists an injection\/ $h\colon A\to B$ such that
	for every\/ $x\in A$\textup, $x\mathrel Rh(x)$\textup.
\item[\textup{(II)}]
	For every\/ $X\subset A$\textup,
	$\norm{R(X)}\ge\norm{X}$\textup.
\end{itemize}
\end{lemma}

\begin{corollary.lemma}
\label{corollary.lemma:1.Halls_lemma.(4.2)}
Let\/ $A$ and\/ $B$ be finite sets\textup,
$R\subset A\times B$\textup,
$f\colon B\to\mathbb N\cup\{0\}$\textup.
Then the following are equivalent\textup.
\begin{itemize}
\item[\textup{(I)}]
	There exists\/ $h\colon A\to B$ such that\/\textup:
	\begin{enumerate}
	\item
		for every\/ $x\in A$\textup, $x\mathrel Rh(x)$\textup, and
	\item
		for every\/ $y\in B$\textup, $\norm{h^{-1}(y)}\le f(y)$\textup.
	\end{enumerate}
\item[\textup{(II)}]
	For every\/ $X\subset A$\textup,
	$$
	\sum_{y\in R(X)}f(y)\ge\norm{X}.
	$$
\end{itemize}
\end{corollary.lemma}

Proofs of the lemma and the corollary may be found, for example,
in~\cite{Muranov:2005:dsmcbgbsg}.

\begin{lemma}[First Estimating Lemma]
\label{lemma:estimating_lemma.selected_arcs.(4.3)}
Let\/ $\Delta$ be a connected\/ $S_1$-map
which either is not elementary spherical or
has a face in whose characteristic boundary
not all nontrivial reduced paths are selected\textup.
Let\/ $A$ be the set of all maximal selected\/
\textup(internal\/\textup) arcs of\/ $\Delta$\textup.
Let\/ $B$ be the set of all faces of\/ $\Delta$ that are incident to
selected arcs\textup.
Let\/ $C$ be some set of faces of\/ $\Delta$ such that
$\Delta$ satisfies the condition\/ $\mathsf Z(2)$
relative to every simple disc submap that
does not contain any faces from\/ $C$ and does not contain
at least one arc from\/ $A$\textup.
Let\/ $n$ be the number of contours of\/ $\Delta$\textup.
Then either\/ $A=\varnothing$\textup, or
$$
\norm{A}\le\sum_{\Pi\in B}\!
\bigl(3+\kappa(\Pi)+\kappa'(\Pi)\bigr)
+2\norm{C\setminus B}-3\chi(\Delta)-n.
$$
Moreover\textup, if\/ $D\subset\Delta(2)$\textup,
then there exist a set\/ $E\subset A$ and a function\/
$f\colon{}\linebreak[0]A\setminus E\to D$ such that\/\textup:
\begin{enumerate}
\item
	either\/ $E=\varnothing$\textup, or
	$$
	\norm{E}\le\sum_{\Pi\in B\setminus D}\!
	\bigl(3+\kappa(\Pi)+\kappa'(\Pi)\bigr)\\
	+2\norm{C\setminus (B\setminus D)}
	-3\chi(\Delta)-n;
	$$
\item
	for every\/ $u\in A\setminus E$\textup, the arc\/ $u$ is incident to
	the face\/ $f(u)$\textup;
\item
	for every\/ $\Pi\in D$\textup,
	$\norm{f^{-1}(\Pi)}\le3+\kappa(\Pi)+\kappa'(\Pi)$\textup;
\item
	for every\/ $\Pi\in C$\textup,
	$\norm{f^{-1}(\Pi)}\le1+\kappa(\Pi)+\kappa'(\Pi)$\textup.
\end{enumerate}
\end{lemma}
\begin{proof}
It suffices to prove this lemma in the case when $\Delta$ is closed
(if $\Delta$ is not closed,
apply this lemma to the closure of $\Delta$ in which
the added ``outer'' faces have no selected paths in their
characteristic boundaries,
to the set $C$ extended by the ``outer'' faces,
and to the same set $D$).
So assume $\Delta$ is closed.
Without loss of generality, assume $D\subset B$.

If $A=\varnothing$, the proof is easy, so assume $A\ne\varnothing$.

Consider an arbitrary non-empty subset $X$ of $A$,
and let $Y$ be the set of all faces of $\Delta$ that are incident to
arcs from $X$ (in particular, $Y\subset B$).
Then
\begin{align*}
&\sum_{\Pi\in Y\cap(D\cap C)}\!\bigl(1+\kappa(\Pi)+\kappa'(\Pi)\bigr)
+\sum_{\Pi\in Y\cap(D\setminus C)}\!
\bigl(3+\kappa(\Pi)+\kappa'(\Pi)\bigr)\\
&\qquad+\sum_{\Pi\in B\setminus D}
\bigl(3+\kappa(\Pi)+\kappa'(\Pi)\bigr)
+2\norm{C\setminus (B\setminus D)}-3\chi(\Delta)\\
&\quad=
\sum_{\Pi\in (Y\cap D)\cap C}\!
\bigl(1+\kappa(\Pi)+\kappa'(\Pi)\bigr)
+\sum_{\Pi\in (Y\cap D)\setminus C}\!
\bigl(3+\kappa(\Pi)+\kappa'(\Pi)\bigr)\\
&\qquad+\sum_{\Pi\in (B\setminus D)\cap C}
\bigl(1+\kappa(\Pi)+\kappa'(\Pi)\bigr)
+\sum_{\Pi\in (B\setminus D)\setminus C}
\bigl(3+\kappa(\Pi)+\kappa'(\Pi)\bigr)\\
&\qquad+2\norm{C}-3\chi(\Delta)\\
&\quad\ge
\sum_{\Pi\in Y\cap C}\!
\bigl(1+\kappa(\Pi)+\kappa'(\Pi)\bigr)
+\sum_{\Pi\in Y\setminus C}\!
\bigl(3+\kappa(\Pi)+\kappa'(\Pi)\bigr)\\
&\qquad+2\norm{C}-3\chi(\Delta)\\
&\quad=
\sum_{\Pi\in Y}\!
\bigl(3+\kappa(\Pi)+\kappa'(\Pi)\bigr)
+2\norm{C\setminus Y}-3\chi(\Delta).
\end{align*}
Thus to complete the proof, it will suffice to show that
$$
\norm{X}\le\sum_{\Pi\in Y}\!
\bigl(3+\kappa(\Pi)+\kappa'(\Pi)\bigr)
+2\norm{C\setminus Y}-3\chi(\Delta),
$$
and afterwards to apply the implication
$\text{(II)}\Rightarrow\text{(I)}$
of Corollary \ref{corollary.lemma:1.Halls_lemma.(4.2)} of Hall's Lemma
to the sets $A$ and $D\sqcup\{\varepsilon\}$, where $\varepsilon$
is some additional element whose preimage is to be taken as the set $E$,
and to the relation $R$ defined as follows:
for all $x\in A$ and $y\in D\sqcup\{\varepsilon\}$,
$x\mathrel Ry$ if and only if either $x$ is incident to
$y\in D$, or $y=\varepsilon$.

Let $K$ be the set of all connected components of the complex
obtained from $\Delta$ by removing
all the faces that are in $Y$ and all the arcs that are in $X$.
For every $\Psi\in K$,
let $d(\Psi)$ denote the number of ends of arcs from $X$
that are attached to $\Psi$.
(One can consider the graph obtained from $K$ and $X$ by collapsing each
elements of $K$ into a point and replacing each arc from $X$
with a single edge---then $d(\Psi)$ will be the degree
of the corresponding vertex of this graph.)
Clearly,
$$
\sum_{\Psi\in K}\chi(\Psi)-\norm{X}+\norm{Y}=\chi(\Delta)
\qquad\text{and}\qquad
\sum_{\Psi\in K}d(\Psi)=2\norm{X}.
$$
These two equations imply that
\begin{align*}
\norm{X}
&=3\norm{Y}+3\sum_{\Psi\in K}\chi(\Psi)-2\norm{X}-3\chi(\Delta)\\
&=3\norm{Y}+\sum_{\Psi\in K}\bigl(3\chi(\Psi)-d(\Psi)\bigr)
-3\chi(\Delta).
\end{align*}

By Lemma \ref{lemma:positive_Euler_characteristic.(3.1)} and Corollary
\ref{corollary.lemma:maps__positive_Euler_characteristic.(3.2)},
for every $\Psi\in K$, $\chi(\Psi)\le1$,
and if $\chi(\Psi)=1$,
then $\Psi$ is the underlying complex of a disc submap of $\Delta$.
Let
$$
K_{i}'=\{\,\Psi\in K\,|\,d(\Psi)=i\ \text{and}\ \chi(\Psi)=1\,\}
\quad\text{for}\quad i=0,1,2,\dots.
$$
Observe that $K_{0}'=\varnothing$.
Then
$$
\norm{X}\le3\norm{Y}+2\norm{K_{1}'}+\norm{K_{2}'}-3\chi(\Delta).
$$
It will suffice to show now that
$$
2\norm{K_{1}'}+\norm{K_{2}'}
\le\sum_{\Pi\in Y}\!(\kappa(\Pi)+\kappa'(\Pi))+2\norm{C\setminus Y}.
$$
Here the condition $\mathsf Z(2)$ shall be used.

Let $L$ be the set of all connected components obtained from
the (disjoint) union of the characteristic boundaries
of all faces from $Y$
by removing all edges and all
intermediate vertices of all their selected paths.
Let $L'$ be the set of those of these components that
either do not consist of
a single vertex or consist of a single vertex whose image
in the $1$-skeleton of $\Delta$ has degree~$1$.
Then
$$
\norm{L}=\sum_{\Pi\in Y}\!\kappa(\Pi),\qquad
\norm{L'}=\sum_{\Pi\in Y}\!\kappa'(\Pi).
$$

The image of each element of $L$ under the corresponding attaching
morphism lies entirely in some element of $K$, hence to each element
of $L$ there is associated an element of $K$.
Because of the condition $\mathsf Z(2)$, and because arcs in $X$
are maximal selected, to each element of
$K_{1}'\sqcup K_{2}'$ not containing any faces from $C\setminus Y$
there is associated at least $1$ element of $L$, and furthermore,
to each element of $K_{1}'$ not containing faces from
$C\setminus Y$
there is associated either at least $2$ element of $L$
or at least $1$ elements of $L'\subset L$.
(A care should be taken when verifying this last statement because
elements of $K_{1}'\sqcup K_{2}'$
are not necessarily underlying complexes
of \emph{simple\/} disc maps, and some may consist
of a single vertex.)
Hence
\begin{align*}
2\norm{K_{1}'}+\norm{K_{2}'}
&\le\norm{L}+\norm{L'}+2\norm{C\setminus Y}\\
&=\sum_{\Pi\in Y}\!(\kappa(\Pi)+\kappa'(\Pi))+2\norm{C\setminus Y}.
\end{align*}
\end{proof}

\begin{lemma}[Second Estimating Lemma]
\label{lemma:estimating_lemma.exceptional_arcs.(4.4)}
Let\/ $\Delta$ be a connected submap of an\/ $S_2$-map\textup,
and suppose\/ $\Delta$ satisfies the condition\/ $\mathsf Y$
relative to the ambient\/ $S_2$-map\textup.
For every\/ $i$\textup,
let\/ $A_i$ be the set of all index-\/$i$
internal exceptional arcs of\/ $\Delta$\textup, and\/
$B_i$ the set of all index-\/$i$ faces of\/ $\Delta$\textup.
For every\/ $i$\textup, let\/ $\delta_i=1$
if\/ $\Delta$ has an external arc incident to a face which is
an exceptional arc of index\/ $i$ in the ambient\/ $S_2$-map\textup,
and let\/ $\delta_i=0$ otherwise\textup.
Then for every\/ $i$\textup, either\/ $A_i=\varnothing$\textup, or
$$
\norm{A_i}\le2\norm{B_i}-\chi(\Delta)-\delta_i.
$$
Furthermore\textup, there exists a set\/ $E$ such that\/\textup:
\begin{enumerate}
\item
	either\/ $E=\varnothing$\textup, or
	$\norm{E}\le-\chi(\Delta)$\textup;
\item
	for every\/ $i$\textup,
	$\norm{A_i\setminus E}\le2\norm{B_i}-\delta_i$\textup.
\end{enumerate}
\end{lemma}
\begin{proof}
For every set $I$,
denote $A_I=\bigcup_{i\in I}A_i$, $B_I=\bigcup_{i\in I}B_i$.
Denote $K_I$ the set of all connected components of
the subcomplex obtained from $\Delta$
by removing all faces that are in $B_I$
and all arcs that are in $A_I$.
Denote $K_I'$ the set of all the elements of $K_I$
of Euler characteristic~$1$.

Let $I$ be an arbitrary set such that $A_I\ne\varnothing$.
It is to be shown that
$$
\norm{A_I}\le2\norm{B_I}-\sum_{i\in I}\delta_i-\chi(\Delta).
$$

By Lemma \ref{lemma:positive_Euler_characteristic.(3.1)} and Corollary
\ref{corollary.lemma:maps__positive_Euler_characteristic.(3.2)},
$\chi(\Psi)\le0$ for each $\Psi\in K_I\setminus K_I'$,
and every element of $K_I'$ is
the underlying complex of a disc submap of $\Delta$.
Therefore,
$$
\chi(\Delta)=\sum_{\Psi\in K_I}\chi(\Psi)-\norm{A_I}+\norm{B_I}
\le\norm{K_I'}-\norm{A_I}+\norm{B_I},
$$
and
$$
\norm{A_I}\le\norm{K_I'}+\norm{B_I}-\chi(\Delta).
$$

Every element of $K_I$ wich is the
underlying complex of a one-contour submap of $\Delta$ is also
an element of $K_{\{i\}}$ for some $i\in I$.
In particular, every element of $K_I'$ is a an element of
$K_{\{i\}}'$ for some $i\in I$.
It follows from the condition $\mathsf Y$ that for every $i$
such that $A_i\ne\varnothing$,
$\norm{K_{\{i\}}'}\le\norm{B_i}-\delta_i$.
(Because if $\Psi\in K_{\{i\}}$ and $\Psi$ contains an index-$i$
exceptional arc of the ambient $S_2$-map which is incident to
a face of $\Delta$, then $\Psi$ is the underlying complex of a submap
with at least $2$ contours,
and hence $\Psi\notin K_{\{i\}}'$.)
Hence
$$
\norm{K_I'}\le\norm{B_I}-\sum_{i\in I}\delta_i,
$$
and
$$
\norm{A_I}\le2\norm{B_I}-\sum_{i\in I}\delta_i-\chi(\Delta).
$$

Now to prove the first part of the satement of the lemma,
it suffices to apply the last inequality to the set $I=\{i\}$,
and to prove the second part
(the existence of the set $E$), it suffices to apply it to
$$
I=\{\,i\,\mid\,\norm{A_i}>2\norm{B_i}-\delta_i\,\}.
$$
\end{proof}

\begin{lemma}[Inductive Lemma]
\label{lemma:inductive_lemma.(4.5)}
Let\/ $\Delta$ be an\/ $S$-map\textup,
and\/ $\Phi$ a simple disc submap of\/ $\Delta$\textup.
Assume that\/ $\Phi$ satisfies the conditions\/ $\mathsf Y$
and\/ $\mathsf D(\lambda,\mu,\nu)$
relative to\/ $\Delta$\textup, where\/
$\lambda,\mu,\nu\colon\Phi(2)\to[0,1]$\textup,
and that
$$
(\lambda+(3+\kappa+\kappa')\mu+2\nu)(\Pi)\le\frac{1}{2}
$$
for every\/ $\Pi\in\Phi(2)$\textup.
Suppose\/ $\Delta$ satisfies the condition\/ $\mathsf Z(2)$
relative to every proper simple disc submap of\/ $\Phi$\textup.
Then\/ $\Delta$ satisfies\/ $\mathsf Z(2)$
relative to\/~$\Phi$\textup.
\end{lemma}
\begin{proof}
This lemma shall be proved by contradiction, so suppose that
$\Delta$ does not satisfy $\mathsf Z(2)$ relative to $\Phi$.

Call an $S_1$-map ``\emph{bad\/}'' if it is elementary spherical
and every nontrivial reduced path in the characteristic boundary of each
of its two faces is selected.
One reason why ``bad'' $S_1$-maps are
inconvenient is that distinct maximal selected arcs in them
can overlap.
More importantly, the First Estimating Lemma does not apply to them.
Observe that $\Delta$ cannot be ``bad.''
Indeed, if $\Delta$ were ``bad,'' then
$\Phi$ would consist of a single face, and for this face $\Pi$,
the conditions $\mathsf D_2(\mu)$ and $\mathsf D_3(\nu)$ would imply
that
$$
\abs{\cntr\Pi}\le(\mu(\Pi)+\nu(\Pi))\abs{\cntr\Pi}
\le\frac{1}{4}\abs{\cntr\Pi}<\abs{\cntr\Pi}.
$$

Let $A$ be the set of all maximal selected arcs of $\Delta$
that are in $\Phi$,
and $B$ the set of all exceptional arcs of $\Delta$
that are in $\Phi$.
Then every element of $B$ is a subarc of an element of~$A$.
Let $L$ be the total number of edges of $\Phi$
that are not edges of elements of $A$,
$M$ be the total number of edges of elements of $A$
that are not edges of elements of $B$, and
$N$ be the total number of edges of all elements
of $B$.
To achieve a contradiction, it shall be shown that
$$
L+M+N\le\frac{1}{2}\sum_{\Pi\in\Phi(2)}\abs{\cntr\Pi}<\norm{\Phi(1)}.
$$

The condition $\mathsf D_1(\lambda)$ gives immediately
the following estimate on $L$:
$$
L\le\sum_{\Pi\in\Phi(2)}\!\lambda(\Pi)\abs{\cntr\Pi}.
$$

To find a good estimate on $M$, consider
the closure $\bar\Phi$ of $\Phi$
($\bar\Phi$ is a spherical map).
Denote $\Theta$ the face of $\bar\Phi$ that is not in $\Phi$
(the ``outer'' face).
Extend the structure of an $S_1$-map from $\Phi$ to $\bar\Phi$
as follows:
select those paths in the characteristic
boundary of $\Theta$ whose images in $\Phi$
coincide with the images of selected paths from the characteristic
boundaries of faces that are in $\Delta(2)\setminus\Phi(2)$.
Then $\bar\Phi$ satisfies $\mathsf Z(2)$
relative to every proper simple disc submap of $\Phi$,
but not relative to $\Phi$ itself, since so does $\Delta$.
In particular, $\kappa(\Theta)\le2$ and $\kappa'(\Theta)=0$.
Observe that $\bar\Phi$ is not ``bad'' because otherwise $\Delta$
would be ``bad'' as well, and that
$A$ is exactly the set of all maximal selected arcs of $\bar\Phi$.
Apply Lemma \ref{lemma:estimating_lemma.selected_arcs.(4.3)}
(First Estimating Lemma)
to $\bar\Phi$,
$\{\Theta\}$ (in the role of the ``set $C$''), and
$\Phi(2)$ (in the role of the ``set $D$'').
Let $f$ be a function $A\to\Phi(2)$ such that:
\begin{enumerate}
\item
    for every $u\in A$, the face $f(u)$ is incident to the arc $u$, and
\item
    for every $\Pi\in\Phi(2)$,
    $\norm{f^{-1}(\Pi)}\le3+\kappa(\Pi)+\kappa'(\Pi)$.
\end{enumerate}
(Since
$3+\kappa(\Theta)+\kappa'(\Pi)-3\chi(\bar\Phi)\le-1\le0$,
the ``set $E$'' is empty.)
By the condition $\mathsf D_2(\mu)$,
$$
M=\sum_{\Pi\in\Phi(2)}\sum_{u:f(u)=\Pi}\!\abs{u}
\le\sum_{\Pi\in\Phi(2)}\!
(3+\kappa(\Pi)+\kappa'(\Pi))\mu(\Pi)\abs{\cntr\Pi}.
$$

It is left to estimate $N$.
Let $B'$ be the set of those arcs from $B$
that are internal in $\Phi$,
and $B''$ the set of those arcs from $B$
that are external in $\Phi$.
For every $i$, let $C_i$ be the set of all index-$i$ faces of $\Phi$,
and $B_i$, $B_i'$, and $B_i''$ be the sets of all index-$i$
elements of $B$, $B'$, and $B''$, respectively.
As in Lemma \ref{lemma:estimating_lemma.exceptional_arcs.(4.4)}
(Second Estimating Lemma),
for every $i$, let $\delta_i=1$ if $B_i''\ne\varnothing$,
and $\delta_i=0$ otherwise.
Let
$$
I=\{\,i\,\mid\,C_i\ne\varnothing\,\}.
$$
By Lemma \ref{lemma:estimating_lemma.exceptional_arcs.(4.4)}
applied to $\Phi$,
$$
\norm{B_i'}\le2\norm{C_i}-1-\delta_i
$$
for every $i\in I$.
It follows now from the condition $\mathsf D_3(\nu)$ that
$$
\sum_{u\in B_i'}\abs{u}
\le(2\norm{C_i}-1-\delta_i)\min_{\Pi\in C_i}\nu(\Pi)\abs{\cntr\Pi},
$$
and
$$
\sum_{u\in B_i''}\abs{u}
\le2\delta_i\min_{\Pi\in C_i}\nu(\Pi)\abs{\cntr\Pi}
$$
for every $i\in I$.
(The constant $2$ in the second inequality
comes from the assumption that
$\Delta$ does not satisfy $\mathsf Z(2)$ relative to $\Phi$.)
Therefore,
$$
\sum_{u\in B_i}\abs{u}
\le2\norm{C_i}\min_{\Pi\in C_i}\nu(\Pi)\abs{\cntr\Pi}
\le\sum_{\Pi\in C_i}2\nu(\Pi)\abs{\cntr\Pi}
$$
for every $i$, and hence
$$
N=\sum_{u\in B}\abs{u}\le\sum_{\Pi\in\Phi(2)}\!2\nu(\Pi)\abs{\cntr\Pi}.
$$

Thus,
\begin{align*}
\norm{\Phi(1)}&=L+M+N\\
&\le\sum_{\Pi\in\Phi(2)}\!
\bigl(\lambda(\Pi)
+(3+\kappa(\Pi)+\kappa'(\Pi))\mu(\Pi)
+2\nu(\Pi)\bigr)\abs{\cntr\Pi}\\
&\le\sum_{\Pi\in\Phi(2)}\frac{1}{2}\abs{\cntr\Pi}
<\norm{\Phi(1)},
\end{align*}
which gives a contradiction.
\end{proof}

The following lemma follows from Lemmas
\ref{lemma:estimating_lemma.selected_arcs.(4.3)},
\ref{lemma:estimating_lemma.exceptional_arcs.(4.4)},
\ref{lemma:inductive_lemma.(4.5)}
and is a slightly simplified version of Lemma 59
in~\cite{Muranov:2007:fgisgicw}.

\begin{lemma}
\label{lemma:useful_lemma.generic.(4.6)}
Let\/ $\Delta$ be a connected\/ $S$-map with\/ $n$ contours such that\/
$n+3\chi(\Delta)\ge0$\textup.
Suppose\/ $\Delta$ satisfies the conditions\/ $\mathsf Y$
and\/ $\mathsf D(\lambda,\mu,\nu)$ relative to itself\textup,
where\/ $\lambda,\mu,\nu\colon\Delta(2)\to[0,1]$\textup.
Let\/ $\gamma=\lambda+(3+\kappa+\kappa')\mu+2\nu$\textup, and
suppose\/ $\gamma(\Pi)\le1/2$ for every\/ $\Pi\in\Delta(2)$\textup.
Then
$$
\sum_i\abs{\cntr_i\Delta}
\ge\sum_{\Pi\in\Delta(2)}(1-2\gamma(\Pi))\abs{\cntr\Pi}.
$$
\end{lemma}



\section{Proof of the main theorem}
\label{section:proof_main_theorem}

In this section, $G$ is the (simple or trivial) group whose
presentation \tGP{a,b}{\mathcal R} is constructed in
Section \ref{section:presentation_construction}.

Every diagram over \tGP{a,b}{\mathcal R} shall be automatically
endowed with a structure of an $S$-map as mentioned in
Section \ref{section:technical_lemmas},
see Figure \ref{figure:4}.
Namely, if $\Delta$ is a diagram over \tGP{a,b}{\mathcal R},
then
\begin{enumerate}
\item
	if $\Pi$ is a face of $\Delta$ and $\ell(\cntr{}\Pi)=r_n$,
	then the index of $\Pi$ shall be $n$, denoted $\iota(\Pi)=n$,
	and a path in the characteristic boundary of $\Pi$
	shall be selected if and only if it is a nontrivial subpath
	of one of the $4k_n$ paths
	corresponding to the subwords
	$u_{n,1}^{\pm1},\dotsc,u_{n,k_n}^{\pm1}$ of $r_n^{\pm1}$;
	note that $\kappa(\Pi)=\kappa_n=2k_n$;
\item
	an arc of $\Delta$ shall be exceptional if and only if
	it is a maximal selected arc which
	corresponds on both sides to the same part
	of some word $u_{n,j}$
	(in particular, the incident faces must have the same index $n$ and
	the label $r_n$).
\end{enumerate}

Call a diagram over \tGP{a,b}{\mathcal R} \emph{convenient\/}
if it is weakly strictly reduced
and every exceptional selected arc in it corresponds to an entire
word $u_{n,j}$ (on both sides).

\begin{lemma}
\label{lemma:convenient_diagrams.(5.1)}
Every diagrammatically reduced diagram
over\/ \tGP{a,b}{\mathcal R}
can be transformed into a convenient one
by a sequence of diamond moves\textup.
\end{lemma}
\begin{proof}
Apply diamond moves to ``extend'' one by one
all exceptional arcs that correspond to proper subwords
of some $u_{n,j}$, and see Lemma \ref{lemma:diamond_move.(3.4)}.
\end{proof}

If $\Delta$ is a diagram over \tGP{a,b}{\mathcal R}, $\bar\Delta$
is the closure of $\Delta$,
and all contour labels of $\Delta$ are cyclically reduced
concatenations of copies of $z_1^{\pm1}$ and $z_2^{\pm1}$,
then the existing structure of an $S$-map
on $\Delta$ shall be automatically
extended to a structure of an $S$-map on $\bar\Delta$
as follows.
If $\Theta$ is any of the ``outer'' faces of $\bar\Delta$
(a face that is not in $\Delta$),
then assign to $\Theta$ the index of $-1$
(to distinguish it from ``inner'' faces)
and choose all nontrivial
reduced paths in the characteristic boundary of $\Theta$ as selected
(note in particular that $\kappa(\Theta)=0$).
There shall be no additional exceptional arcs in~$\bar\Delta$.

If $\Delta$ is a diagram over \tGP{a,b}{\mathcal R},
then let functions
$\lambda,\mu,\nu,\gamma\colon\Delta(2)\to[0,1]$ be
defined as follows:
\begin{gather*}
\lambda(\Pi)=\lambda_{\iota(\Pi)},\qquad
\mu(\Pi)=\mu_{\iota(\Pi)},\qquad
\nu(\Pi)=\nu_{\iota(\Pi)},\\
\gamma(\Pi)
=\lambda(\Pi)+(3+2\kappa(\Pi))\mu(\Pi)+2\nu(\Pi)
=\gamma_{\iota(\Pi)}
\end{gather*}
for every $\Pi\in\Delta(2)$.
Observe that
$$
\gamma_\Delta(\Pi)<\frac{1}{2}
$$
for every $\Pi\in\Delta(2)$
by condition
(C\ref{item:presentation_construction.main_properties.8})
in Section~\ref{section:presentation_construction}.

\begin{lemma}
\label{lemma:condition_D.(5.2)}
Every convenient diagram\/ $\Delta$
over\/ \tGP{a,b}{\mathcal R}
satisfies the condition\/
$\mathsf D(\lambda,\mu,\nu)$
relative to itself\textup.
If all contour labels of\/ $\Delta$
are cyclically reduced concatenations
of copies of\/ $z_1^{\pm1}$ and\/ $z_2^{\pm1}$\textup,
then\/ $\Delta$ also satisfies\/
$\mathsf D(\lambda,\mu,\nu)$
relative to its closure\textup.
\end{lemma}
\begin{proof}
This lemma follows from conditions
(C\ref{item:presentation_construction.main_properties.3}),
(C\ref{item:presentation_construction.main_properties.4}),
(C\ref{item:presentation_construction.main_properties.5}),
(C\ref{item:presentation_construction.main_properties.6})
in Section~\ref{section:presentation_construction}.
\end{proof}

\begin{lemma}
\label{lemma:condition_Y.(5.3)}
Let\/ $\Delta$ be a convenient connected
subdiagram of some diagram
over\/ \tGP{a,b}{\mathcal R}\textup, and let\/ $J$
be the set of indices of faces of\/ $\Delta$
\textup($J\subset I$\textup)\textup.
Suppose that for every proper subset\/ $J'\subsetneqq J$\textup,
the presentation\/ \tGP{a,b}{r_i,\ i\in J'}
defines a torsion-free group\textup.
Then\/ $\Delta$ satisfies the condition\/ $\mathsf Y$
relative to the ambient diagram\textup.
\end{lemma}
\begin{proof}
Without loss of generality (or inducting on $J$),
assume that for every proper subset $J'\subsetneqq J$,
every convenient connected diagram over
\tGP{a,b}{r_i,\ i\in J'}
satisfies $\mathsf Y$ relative to itself.

Let $i$ be the index of an arbitrary internal exceptional arc of $\Delta$.
Let $A_i$ be the set of all internal exceptional arcs of $\Delta$
of index $i$,
and $B_i$ be the set of all faces of $\Delta$ of index $i$.
Let $K_i$ be the set of all connected components of
the subcomplex obtained from $\Delta$
by removing all faces that are in $B_i$
and all arcs that are in $A_i$.
View each element of $K_i$ as a subdiagram of~$\Delta$.

Let
$\hat J_i=J\setminus\{i\}$.
Then the group presented by \tGP{a,b}{r_i,\ i\in \hat J_i} is torsion-free,
and every element of $K_i$ is a diagram
over \tGP{a,b}{r_i,\ i\in \hat J_i}.

Let $L$ be the set of vertices of the characteristic boundaries
of index-$i$ faces chosen as follows:
if $\Pi\in B_i$, then let
$s_1,\dotsc,s_{k_i},s_1',\dotsc,s_{k_i}',t_1,\dotsc,t_{k_i},t_0$
be the paths in the characteristic boundary of $\Pi$
such that the concatenation
$s_1t_1s_1's_2t_2s_2'\dotsm s_{k_i}t_{k_i}s_{k_i}'t_0$
is the characteristic contour of $\Pi$, and
the label of $s_j$ is $u_{i,j}$ and
the label of $s_j'$ is $u_{i,j}^{-1}$ for $j=1,\dotsc,k_i$.
Let $L$ contain the initial vertices of the $k_i$ paths
$s_1t_1s_1',\dotsc,s_{k_i}t_{k_i}s_{k_i}'$ and no other
vertices of the characteristic boundary of $\Pi$.
Note that $\norm{L}=k_i\norm{B_i}$.

The image (under the corresponding attaching morphism)
of each element of $L$
is a vertex of an element of $K_i$.
It suffices to prove now that for every $\Psi\in K_i$
such that $\chi(\Psi)=1$ and
for every $\Psi\in K_i$ such that $\Psi$ contains an index-$i$
exceptional arc of the ambient diagram incident to a face of $\Delta$,
the number of elements of $L$ whose images are in $\Psi$ is
at least~$k_i$.

It follows from the definition of
exceptional arcs in a convenient diagram that for every
$\Psi\in K_i$, the number of elements of $L$ whose images are in $\Psi$
is divisible by $k_i$.
An example of a possible situation is
shown on Figure \ref{figure:5}
(under the assumption that $i=1$ and $k_1=4$),
elements of $L$ are marked there with thick dots.
\begin{figure}\centering
\includegraphics[scale=1.1]{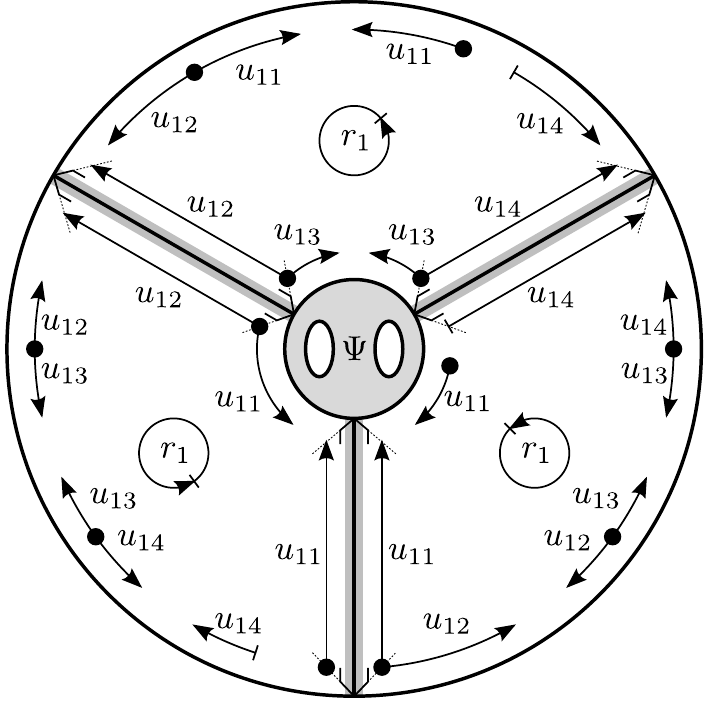}
\caption{Counting elements of $L$ mapped to $\Psi$
(there are $4$ such elements here, $4$ is divisible by $k_1=4$).}
\label{figure:5}
\end{figure}

It is left to show that if $\Psi\in K_i$
and either $\chi(\Psi)=1$ or $\Psi$ contains an
index-$i$ exceptional arc incident to a face of $\Delta$, then
the set of elements of $L$ mapped to $\Psi$ is non-empty
(and hence there are at least $k_i$ such elements).

Suppose $\Psi\in K_i$ and $\Psi$ contains an index-$i$
exceptional arc $u$ incident to a face $\Pi$ of $\Delta$.
(This arc is an external arc of $\Delta$.)
Let $s_1,\dotsc,s_{k_i}$, $s_1',\dotsc,s_{k_i}'$,
$t_1,\dotsc,t_{k_i}$, $t_0$
be the paths in the characteristic boundary of $\Pi$
as in the definition of $L$.
Then $u$ lies on the image of one of the paths
$s_1,\dotsc,s_{k_i}$, $s_1',\dotsc,s_{k_i}'$.
If $u$ lies on the image of $s_j$ for some $j=1,\dotsc,k_i$,
then the initial vertex of $s_j$ is an elements of $L$ and is mapped
to $\Psi$.
If $u$ lies on the image of $s_j'$ for some $j=1,\dotsc,k_i-1$,
then the initial vertex of $s_{j+1}$
is an elements of $L$ and is mapped to $\Psi$.
If $u$ lies on the image of $s_{k_i}'$,
then the initial vertex of $s_1$
is an elements of $L$ and is mapped to $\Psi$.
Thus in this case the set of elements of $L$ mapped to $\Psi$
is non-empty.

Now consider an arbitrary element $\Psi$ of $K_i$
of Euler characteristic~$1$.
By Corollary
\ref{corollary.lemma:maps__positive_Euler_characteristic.(3.2)},
$\Psi$ is a disc subdiagram of $\Delta$.
Suppose that no element of $L$ is mapped to $\Psi$,
as on Figure \ref{figure:6}, for example.
After cyclically shifting and/or inverting the contour of $\Psi$
if necessary,
it can be concluded that
the contour label of $\Psi$ is of the form $w^n$ for some $n\ne0$,
where $w=w_i$ if $r_i$ is a relator of the first kind,
and $w=v$ if $r_i$ is a relator of the second kind
(as defined in Section \ref{section:presentation_construction}).
Therefore $[w]^n=1$ in the group presented by
\tGP{a,b}{r_i,\ i\in \hat J_i},
and hence $[w]=1$ in this group, since it is torsion-free.
\begin{figure}\centering
\includegraphics[scale=1.1]{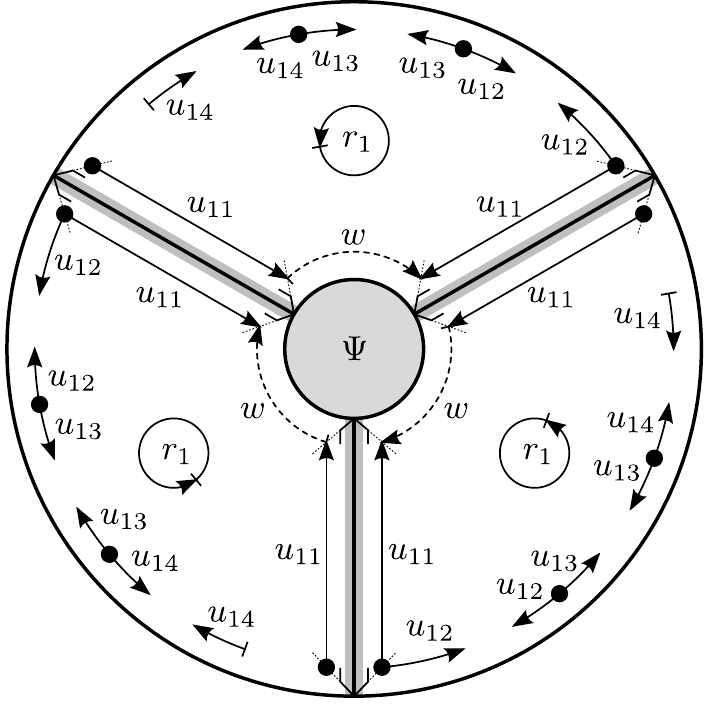}
\caption{A situation prohibited by the absence of torsion
(if $w^3=1$, then $w=1$).}
\label{figure:6}
\end{figure}

Let $\Gamma$ be a convenient disc diagram
over \tGP{a,b}{r_i,\ i\in \hat J_i}
with contour label $w$ (see Lemma \ref{lemma:diagrams.(3.6)}).
By Lemma \ref{lemma:condition_D.(5.2)},
the diagram $\Gamma$ satisfies
the condition $\mathsf D(\lambda,\mu,\nu)$
relative to itself.
By inductive assumption, $\Gamma$ satisfies
the condition $\mathsf Y$ relative to itself.
Therefore Lemma \ref{lemma:useful_lemma.generic.(4.6)} can be applied,
and hence
$$
\abs{w}=\abs{\cntr{}\Gamma}
\ge\sum_{\Pi\in\Gamma(2)}(1-2\gamma(\Pi))\abs{\cntr\Pi}.
$$
Combined with
condition (C\ref{item:presentation_construction.main_properties.10}),
this inequality excludes the case $w=v$,
and in the case $w=w_i$ it shows, using
condition (C\ref{item:presentation_construction.main_properties.9}),
that $\iota(\Pi)<i$ for every $\Pi\in\Gamma(2)$.
Hence $[w_i]=1$ in the group presented by \tGP{a,b}{r_n,\ n\in I, n<i},
but this contradicts the construction of the index set~$I$.
\end{proof}

\begin{lemma}
\label{lemma:condition_Y__singular_asphericity.(5.4)}
Let\/ $J$ be a finite subset of\/ $I$\textup.
Then every convenient connected diagram
over\/ \tGP{a,b}{r_i,\ i\in J}
satisfies the condition\/ $\mathsf Y$ relative to itself\textup,
the presentation\/ \tGP{a,b}{r_i,\ i\in J} is singularly aspherical\textup,
and the group presented by\/ \tGP{a,b}{r_i,\ i\in J}
is torsion-free\textup.
\end{lemma}
\begin{proof}
Without loss of generality (or inducting of $J$),
assume that all proper
subpresentations of \tGP{a,b}{r_i,\ i\in J}
define torsion-free groups.

Then, by Lemma \ref{lemma:condition_Y.(5.3)},
every convenient connected diagram over
\tGP{a,b}{r_i,\ \linebreak[3]i\in J}
satisfies the condition $\mathsf Y$
relative to itself.

To show that \tGP{a,b}{r_i,\ i\in J} is singularly aspherical,
it suffices to show that it is diagrammatically aspherical
and to recall conditions
(C\ref{item:presentation_construction.main_properties.11}) and
(C\ref{item:presentation_construction.main_properties.12}).

Suppose \tGP{a,b}{r_i,\ i\in J} is not diagrammatically aspherical.
Then there exists a convenient spherical diagram over
\tGP{a,b}{r_i,\ i\in J} with at least $3$ faces
(use Lemmas
\ref{lemma:diamond_move.(3.4)},
\ref{lemma:diagram_reduction.(3.5)},
\ref{lemma:convenient_diagrams.(5.1)});
let $\Delta$ be such a diagram.
Then $\Delta$ satisfies the condition $\mathsf Y$
relative to itself.
By Lemma \ref{lemma:condition_D.(5.2)}, $\Delta$ satisfies
$\mathsf D(\lambda,\mu,\nu)$ relative to itself.
Hence, by Lemma \ref{lemma:useful_lemma.generic.(4.6)}
and
condition (C\ref{item:presentation_construction.main_properties.8}),
$$
0\ge\sum_{\Pi\in\Delta(2)}
\bigl(1-2\gamma(\Pi)\bigr)\abs{\cntr\Pi}>0,
$$
which gives a contradiction.

Thus \tGP{a,b}{r_i,\ i\in J} is singularly aspherical,
and it defines a torsion-free group
by Lemma~\ref{lemma:singularly_aspherical_groups.(3.8)}.
\end{proof}

\begin{corollary.lemma}
\label{corollary.lemma:condition_Y__singular_asphericity.(5.5)}
Every convenient connected diagram
over\/ \tGP{a,b}{\mathcal R}
satisfies the condition\/ $\mathsf Y$ relative to itself\textup.
The presentation\/ \tGP{a,b}{\mathcal R} is singularly aspherical\textup.
The group\/ $G$ is torsion-free\textup.
\end{corollary.lemma}

\begin{proposition}
\label{proposition:torsion-free.(5.6)}
The group\/ $G$ is torsion-free and of cohomological and
geometric dimension~$2$\textup.
\end{proposition}
\begin{proof}
This follows from
Lemma \ref{lemma:singularly_aspherical_groups.(3.8)} and
Corollary~\ref{corollary.lemma:condition_Y__singular_asphericity.(5.5)}.
\end{proof}

\begin{lemma}
\label{lemma:useful_lemma.particular.(5.7)}
If\/ $\Delta$ is a convenient disc or annular
diagram over\/ \tGP{a,b}{\mathcal R}\textup, then
$$
\sum_{\Pi\in\Delta(2)}(1-2\gamma(\Pi))\abs{\cntr\Pi}
\le\sum_i\abs{\cntr_i\Delta}.
$$
\end{lemma}
\begin{proof}
This lemma follows from Lemmas
\ref{lemma:useful_lemma.generic.(4.6)}, \ref{lemma:condition_D.(5.2)},
and Corollary
\ref{corollary.lemma:condition_Y__singular_asphericity.(5.5)}.
\end{proof}

\begin{proposition}
\label{proposition:limit_hyperbolic_groups.(5.8)}
Every finite subpresentation of\/ \tGP{a,b}{\mathcal R}
presents a hyperbolic group\textup.
\end{proposition}
\begin{proof}
Let $J$ be a finite subset of $I$.
Let $q=\min_{i\in J}(1-2\gamma_i)$.
Then $q>0$
by condition (C\ref{item:presentation_construction.main_properties.8}).

A group is hyperbolic if and only if it has
a finite presentation with a linear \emph{isoperimetric function},
see Theorems 2.5, 2.12 in \cite{AlonsoBCFLMSS:1991:nwhg}.

Consider an arbitrary $w\in\{a^{\pm1},b^{\pm1}\}^*$
such that $[w]=1$ in the group presented by \tGP{a,b}{r_i,\ i\in J}.
Take an arbitrary diamond-move reduced
disc diagram over \tGP{a,b}{r_i,\ i\in J}
with contour label $w$
and transform it by diamond moves into a convenient
disc diagram $\Delta$ with contour label $w$
(see Lemmas
\ref{lemma:diamond_move.(3.4)},
\ref{lemma:diagrams.(3.6)},
\ref{lemma:convenient_diagrams.(5.1)}).
Then, by Lemma \ref{lemma:useful_lemma.particular.(5.7)},
$$
\norm{\Delta(2)}\le\sum_{\Pi\in\Delta(2)}\abs{\cntr\Pi}
\le\frac{1}{q}\abs{\cntr{}\Delta}.
$$
Thus the isoperimetric function of
\tGP{a,b}{r_i,\ i\in J} is linear.
\end{proof}

\begin{proposition}
\label{proposition:free_subgroups.infinite_commutator_width.(5.9)}
The elements\/ $[z_1]$ and\/ $[z_2]$
freely generate a free subgroup\/ $H$ of\/ $G$ such that
for every\/ $h\in H\setminus\{1\}$\textup,
$$
\lim_{n\to\infty}\cl_{G}(h^n)=\infty.
$$
\end{proposition}
\begin{proof}
Let $w$ be a nontrivial reduced concatenation of
several copies of $z_1^{\pm1}$ and $z_2^{\pm1}$, and
$n$ a positive integer.
To prove this lemma, it suffices to show that
$\cl_G([w^m])>n$ for every large enough $m$.
Without loss of generality, assume that $w$ is cyclically reduced.

Denote $c_n=2-2n$
(the Euler characteristic of a sphere with $n$ handles).
Let
$$
J_n=\{\,i\in I\,\mid\,\chi_i>c_n\,\}.
$$
The set $J_n$ is finite by
condition (C\ref{item:presentation_construction.main_properties.7}).
Let
$$
p_n=\max_{i\in J_n}\mu_i\abs{r_i},
\qquad
q_n=\max_{i\in J_n}\nu_i\abs{r_i}.
$$
Let $m$ be an arbitrary integer such that
$$
(3-3c_n)\max\{p_n,\abs{w}-1\}+(1-c_n)q_n<\frac{1}{2}\abs{w^m}.
$$
It shall now suffice to show that $\cl_G([w^m])>n$.
This can be done by contradiction.

Suppose $\cl_G([w^m])\le n$.
Then, by Lemmas
\ref{lemma:diamond_move.(3.4)},
\ref{lemma:diagrams.(3.6)},
\ref{lemma:convenient_diagrams.(5.1)},
there exists
a convenient one-contour diagram $\Delta$ over
\tGP{a,b}{\mathcal R} with contour label $w^m$ and
whose closure is a combinatorial sphere with $n$ or fewer handles.
Let $\bar\Delta$ be the closure of $\Delta$.
Then $\chi(\bar\Delta)\ge c_n$.

Let $\Theta$ be the ``outer'' face of $\bar\Delta$
(the face that is not in $\Delta$).
Then $\kappa(\Theta)=0$
(recall that a structure of an $S$-map is defined on~$\bar\Delta$).

By Lemma \ref{lemma:condition_D.(5.2)},
$\Delta$ satisfies
$\mathsf D(\lambda,\mu,\nu)$
relative to $\bar\Delta$.
By Corollary \ref{corollary.lemma:condition_Y__singular_asphericity.(5.5)},
$\Delta$ satisfies $\mathsf Y$ relative to itself and hence
relative to $\bar\Delta$ as well.
By induction and Lemma \ref{lemma:inductive_lemma.(4.5)} (Inductive Lemma),
using condition
(C\ref{item:presentation_construction.main_properties.8}),
obtain that $\bar\Delta$ satisfies $\mathsf Z(2)$ relative
to every simple disc subdiagram of~$\Delta$.

Let $N$ be the sum of the lengths of all exceptional arcs
of $\bar\Delta$,
$M$ the sum of the lengths of all non-exceptional maximal selected
arcs of $\bar\Delta$, and
$L$ the number of edges of $\bar\Delta$ that do not lie
on any selected arc.
To obtain a contradiction, it suffices to show that
$$
L+M+N<\frac{1}{2}\sum_{\Pi\in\bar\Delta(2)}\abs{\cntr\Pi}.
$$

The following upper estimate on $L$ follows from the condition
$\mathsf D_1(\lambda)$ satisfied by $\Delta$ relative to $\bar\Delta$:
$$
L\le\sum_{\Pi\in\Delta(2)}\!\lambda(\Pi)\abs{\cntr\Pi}
$$
(because every edge of the characteristic boundary of $\Theta$ lies
on a selected path).

Let $A$ be the set of all maximal selected arcs of $\bar\Delta$.
Then
$$
\sum_{u\in A}\abs{u}=M+N.
$$

Apply Lemma \ref{lemma:estimating_lemma.selected_arcs.(4.3)}
(First Estimating Lemma)
to $\bar\Delta$ and the sets
$C=\{\Theta\}$ and $D=\Delta(2)$.
Let $E$ be a subset of $A$ and $h$ a function
$A\setminus E\to\Delta(2)$ such that:
\begin{enumerate}
\item
    either $E=\varnothing$, or
    $\norm{E}\le3-3\chi(\bar\Delta)\le3-3c_n$;
\item
    for every arc $u\in A\setminus E$, the face $h(u)$ is incident to $u$;
\item
    for every face $\Pi\in\Delta(2)$,
    $\norm{h^{-1}(\Pi)}\le3+2\kappa(\Pi)$.
\end{enumerate}
Let $M_1$ be the sum of the lengths of all non-exceptional
arcs from $A\setminus E$,
and $M_2$ the sum of the lengths of all non-exceptional
arcs from $E$.
Then $M_1+M_2=M$, and, by the condition $\mathsf D_2(\mu)$,
$$
M_1
\le\sum_{\Pi\in\Delta(2)}\!
(3+2\kappa(\Pi))\mu(\Pi)\abs{\cntr\Pi},
$$
while
$$
M_2
\le(3-3c_n)p_\Delta,
$$
where $p_\Delta$ is
the maximal length of a non-exceptional maximal selected arc of $\Delta$.

Apply Lemma \ref{lemma:estimating_lemma.exceptional_arcs.(4.4)}
(Second Estimating Lemma) to $\Delta$.
Let $F$ be a set of exceptional arcs of $\Delta$ such that:
\begin{enumerate}
\item
    either $F=\varnothing$, or
    $\norm{F}\le-\chi(\Delta)\le1-c_n$, and
\item
    for every $i$,
    the number of exceptional arcs of index $i$ that are not in $F$
    is at most twice the number of faces of $\Delta$ of index $i$.
\end{enumerate}
Let $N_1$ be the sum of the lengths of all exceptional arcs that
are not in $F$,
and $N_2$ the sum of the lengths of all arcs in $F$.
Then $N_1+N_2=N$, and, by the condition $\mathsf D_3(\nu)$,
$$
N_1\le\sum_{\Pi\in\Delta(2)}\!2\nu(\Pi)\abs{\cntr\Pi},
$$
while
$$
N_2\le(1-c_n)q_\Delta,
$$
where $q_\Delta$ is
the maximal length of an exceptional arc of $\Delta$.

Adding up the estimates for $L$, $M_1$, and $N_1$ gives:
$$
L+M_1+N_1\le\sum_{\Pi\in\Delta(2)}\!\gamma(\Pi)\abs{\cntr\Pi}.
$$
Now $M_2+N_2$ is to be estimated by estimating $p_\Delta$ and $q_\Delta$.

The length of every arc of $\bar\Delta$ that is
incident to $\Theta$ and not incident to any other face
is at most $\abs{w}-1$.
Indeed, the label of each of the two
oriented arcs associated with such an arc
is a common subword of $w^m$ and $w^{-m}$
(because $\bar\Delta$ is orientable),
and
any such word of length $\abs{w}$ would be simultaneously a cyclic
shift of $w$ and of $w^{-1}$, but
in a free group, no nontrivial element can be conjugate to its own inverse
(if it was, it would commute with the square of the conjugating element,
and hence with the conjugating element itself).

Let
$$
Q_{\Delta,n}=\{\,\Pi\in\Delta(2)\,\mid\,\iota(\Pi)\notin J_n\,\}
=\{\,\Pi\in\Delta(2)\,\mid\,\chi_{\iota(\Pi)}\le c_n\,\}.
$$
Then, as follows from conditions
(C\ref{item:presentation_construction.main_properties.4}),
(C\ref{item:presentation_construction.main_properties.5}),
(C\ref{item:presentation_construction.main_properties.6}),
and the estimate $\abs{w}-1$ on the lengths of arcs incident only
to $\Theta$,
\begin{align*}
p_\Delta&\le\max\bigl(\,\{\,\mu(\Pi)\abs{\cntr\Pi}\,
\mid\,\Pi\in Q_{\Delta,n}\,\}
\cup\{p_n,\abs{w}-1\}\,\bigr),\\
q_\Delta&\le\max\bigl(\,\{\,\nu(\Pi)\abs{\cntr\Pi}\,
\mid\,\Pi\in Q_{\Delta,n}\,\}
\cup\{q_n\}\,\bigr).
\end{align*}
Therefore,
\begin{align*}
M_2&\le\sum_{\Pi\in Q_{\Delta,n}}\!(3-3c_n)\mu(\Pi)\abs{\cntr\Pi}
+(3-3c_n)\max\{p_n,\abs{w}-1\},\\
N_2&\le\sum_{\Pi\in Q_{\Delta,n}}\!(1-c_n)\nu(\Pi)\abs{\cntr\Pi}
+(1-c_n)q_n,
\end{align*}
and hence, by the choice of $m$,
$$
M_2+N_2
<\sum_{\Pi\in\Delta(2)}\!
\bigl((3-3\chi_{\iota(\Pi)})\mu(\Pi)
+(1-\chi_{\iota(\Pi)})\nu(\Pi)\bigr)\abs{\cntr\Pi}
+\frac{1}{2}\abs{\cntr\Theta}.
$$

Adding up the obtained estimates for $L+M_1+N_1$ and $M_2+N_2$
and using condition
(C\ref{item:presentation_construction.main_properties.8}) gives
a contradiction:
$$
\norm{\Delta(1)}
=L+M+N
<\sum_{\Pi\in\Delta(2)}\frac{1}{2}\abs{\cntr\Pi}
+\frac{1}{2}\abs{\cntr\Theta}
=\norm{\Delta(1)}.
$$
\end{proof}

\begin{proposition}
\label{proposition:infinite_square_width.(5.10)}
The function\/ $\sql_G$ is unbounded on the free subgroup\/
$H=\langle[z_1],[z_2]\rangle$\textup.
\end{proposition}
\begin{proof}
Let $n$ be an arbitrary positive integer,
$c_n=2-n$, and
$$
J_n=\{\,i\in I\,\mid\,\chi_i>c_n\,\}.
$$
The set $J_n$ is finite
by condition
(C\ref{item:presentation_construction.main_properties.7}).
Let
$$
p_n=\max_{i\in J_n}\mu_i\abs{r_i},
\qquad
q_n=\max_{i\in J_n}\nu_i\abs{r_i}.
$$

Let $w$ be a cyclically reduced non-periodic
concatenation of
several copies of $z_1$ and $z_2$ such that
if $s$ is a common prefix of two distinct cyclic shifts of $w$,
then
$$
(3-3c_n)\max\{p_n,\abs{s}\}+(1-c_n)q_n<\frac{1}{2}\abs{w}
$$
(see properties of $z_1$, $z_2$ in
Section \ref{section:presentation_construction}.)
Note that $w$ is a positive word, and hence $w$ and $w^{-1}$
have no nontrivial common subwords.

Suppose $\sql_G([w])\le n$.
Then there exists
a convenient connected one-contour diagram $\Delta$ over
\tGP{a,b}{\mathcal R} such that $\chi(\Delta)\ge c_n-1$ and
$\ell(\cntr{}\Delta)=w$,
see Lemmas
\ref{lemma:diamond_move.(3.4)},
\ref{lemma:diagrams.(3.6)},
\ref{lemma:convenient_diagrams.(5.1)}.

Let $\bar\Delta$ be the closure of $\Delta$ and $\Theta$
be the ``outer'' face.
It follows from the choice of $w$ that if $u$ is an arc of $\bar\Delta$
which is incident to $\Theta$ and not incident to any other face,
then
$$
(3-3c_n)\abs{u}+(1-c_n)q_n<\frac{1}{2}\abs{\cntr\Theta}.
$$

Now virtually the same argument of counting all edges of $\Delta$
in two different ways as in the proof of Proposition
\ref{proposition:free_subgroups.infinite_commutator_width.(5.9)}
leads to a contradiction and completes the proof.
\end{proof}

\begin{proposition}
\label{proposition:solvable_conjugacy_problem.(5.11)}
The construction of the presentation\/ \tGP{a,b}{\mathcal R}
in Section\/ \textup{\ref{section:presentation_construction}}
can be carried out in such a way that\/
$G$ have decidable word and conjugacy problems\textup.
\end{proposition}
\begin{proof}
Let the construction of the presentation \tGP{a,b}{\mathcal R}
be carried out in such a way that the conditions in the statement of
Lemma \ref{lemma:effective_presentation.(2.2)} be
satisfied (use Lemma \ref{lemma:effective_presentation.(2.2)}).

It shall be shown first that there exists an algorithm that solves
the word problem for all finite subpresentations
of \tGP{a,b}{\mathcal R}.
More precisely, it needs to be shown that there exists an algorithm
which for every input $(\mathcal S,w)$ with
$\mathcal S=\{r_n\}_{n\in J}$, $J\subset I$, $J$ finite,
and $w\in\{a^{\pm1},b^{\pm1}\}^*$,
decides correctly whether $[w]=1$
in the group presented by \tGP{a,b}{\mathcal S}.
The following algorithm does it:
\begin{enumerate}
\renewcommand{\labelenumi}{\theenumi.}
\item
	input: $\mathcal S$, $w$;
\item
	construct all disc diagrams $\Delta$ over
	\tGP{a,b}{\mathcal S}
	satisfying
	$$
	\sum_{\Pi\in\Delta(2)}(1-2\gamma(\Pi))\abs{\cntr\Pi}
	\le\abs{\cntr{}\Delta}=\abs{w}
	$$
	(there are only finitely many such diagrams, and they all
	can be effectively constructed);
\item
	if among the constructed diagrams there is one
	with contour label $w$, then the output is
	``\texttt{yes}''
	(meaning $[w]=1$ in the group presented by \tGP{a,b}{\mathcal S});
	otherwise the output is ``\texttt{no}.''
\end{enumerate}
It follows from Lemmas
\ref{lemma:diamond_move.(3.4)},
\ref{lemma:diagrams.(3.6)},
\ref{lemma:convenient_diagrams.(5.1)},
\ref{lemma:useful_lemma.particular.(5.7)}
that the described algorithm works correctly.
Hence, by Lemma \ref{lemma:effective_presentation.(2.2)},
the set $I$ and the family $\mathcal R$ are recursive.

The following algorithm decides the word problem in \tGP{a,b}{\mathcal R},
\ie, given $w\in\{a^{\pm1},b^{\pm1}\}^*$,
it decides whether
$[w]=1$ in the group presented by \tGP{a,b}{\mathcal R}:
\begin{enumerate}
\renewcommand{\labelenumi}{\theenumi.}
\item
	input: $w$;
\item
	find an $n$ such that
	$(1-2\gamma_i)\abs{r_i}>\abs{w}$ for every $i>n$
	(see condition (\ref{item:lemma.effective_presentation.2})
	of the statement of Lemma \ref{lemma:effective_presentation.(2.2)});
\item
	construct all disc diagrams $\Delta$ over
	\tGP{a,b}{r_i,\ i\in I, i\le n}
	satisfying
	$$
	\sum_{\Pi\in\Delta(2)}(1-2\gamma(\Pi))\abs{\cntr\Pi}
	\le\abs{\cntr{}\Delta}=\abs{w};
	$$
\item
	if among the constructed diagrams there is one
	with contour label $w$, then the output is
	``\texttt{yes}''
	(meaning $[w]=1$ in the group presented by \tGP{a,b}{\mathcal R});
	otherwise the output is ``\texttt{no}.''
\end{enumerate}

The following algorithm decides the conjugacy problem
in \tGP{a,b}{\mathcal R},
\ie, given $w_1$ and $w_2$, whether
$[w_1]$ and $[w_2]$ are conjugate in the group
presented by \tGP{a,b}{\mathcal R}:
\begin{enumerate}
\renewcommand{\labelenumi}{\theenumi.}
\item
	input: $w_1$, $w_2$;
\item
	decide, using the previous algorithm, whether $[w_1]=1$
	and whether $[w_2]=1$
	in the group presented by \tGP{a,b}{\mathcal R};
\item
	if $[w_1]=1$ or $[w_2]=1$, then output the appropriate answer
	(whether or not $[w_1]$ and $[w_2]$ are conjugate)
	and stop here;
	otherwise continue;
\item
	find an $n$ such that
	$(1-2\gamma_i)\abs{r_i}>\abs{w_1}+\abs{w_2}$ for every $i>n$;
\item
	construct all oriented annular diagrams $\Delta$ over
	\tGP{a,b}{r_i,\ i\in I, i\le n}
	whose contours agree with the orientation and which
	satisfy
	\begin{gather*}
	\abs{\cntr_1\Delta}=\abs{w_1},
	\qquad
	\abs{\cntr_2\Delta}=\abs{w_2},\\
	\sum_{\Pi\in\Delta(2)}(1-2\gamma(\Pi))\abs{\cntr\Pi}
	\le\abs{\cntr_1\Delta}+\abs{\cntr_2\Delta};
	\end{gather*}
\item
	if among the constructed diagrams there is one
	with contour labels $w_1$ and $w_2^{-1}$, respectively,
	then the output is
	``\texttt{yes}''
	(meaning $[w_1]$ and $[w_2]$ are conjugate in the group
	presented by \tGP{a,b}{\mathcal R});
	otherwise the output is ``\texttt{no}.''
\end{enumerate}

(Lemmas
\ref{lemma:diamond_move.(3.4)},
\ref{lemma:diagrams.(3.6)},
\ref{lemma:convenient_diagrams.(5.1)},
\ref{lemma:useful_lemma.particular.(5.7)}
shall be used to verify correctness of the last two algorithms.)
\end{proof}

The main theorem follows from Propositions
\ref{proposition:simple.norms_stably_bounded.(2.1)},
\ref{proposition:torsion-free.(5.6)},
\ref{proposition:limit_hyperbolic_groups.(5.8)},
\ref{proposition:free_subgroups.infinite_commutator_width.(5.9)},
\ref{proposition:infinite_square_width.(5.10)},
\ref{proposition:solvable_conjugacy_problem.(5.11)}.



\bibliographystyle{amsplain}
\bibliography{\string~/Documents/My_Math_Papers/bib}

\end{document}